\documentclass[12pt]{amsart}
\usepackage{a4,graphicx}
\usepackage{xcolor,enumerate}
\usepackage[pagewise]{lineno}
\usepackage{mathrsfs}
\usepackage[linktoc=all]{hyperref}
\usepackage[top=2.cm,bottom=2.cm,left=2.cm, right=2.cm]{geometry} 

\let\eps\varepsilon  
\newcommand{\N}{{\mathbb N}}  
  
\newcommand{\R}{{\mathbb R}} 
 
\newcommand{\mub}{\vec\mu}

\newcommand{\bn}{\boldsymbol{n}}
\newcommand{\sspan}{\operatorname{span}}

\newcommand{\Ib}{{\mathbb I}}  
\newcommand{\dv}{\operatorname{div}}  
\newcommand{\curl}{\operatorname{curl}}

\newcommand{\Sn}{S^{(n)}}
\newcommand{\rhov}{\vec{\rho}}
\newcommand{\rhon}{\rho^{(n)}}
\newcommand{\rhonv}{\vec{\rho}^{(n)}}

\newcommand{\pn}{p^{(n)}}
\newcommand{\Tn}{T^{(n)}}
\newcommand{\vn}{\mathbf{v}^{(n)}}
\newcommand{\Jn}{\mathbf{J}^{(n)}}
\newcommand{\sgn}{\sigma^{(n)}}

\newcommand{\veb}[1]{\mathbf{#1}}

\newcommand{\tf}{\mathcal{T}}
\newcommand{\Qt}{Q_\mathcal{T}}
\newcommand{\uml}{\"}
\newcommand{\chf}[1]{{\raisebox{3pt}{\Large $\chi$}}_{#1}}
\let\pa\partial

\newtheorem{theorem}{Theorem}   
\newtheorem{lemma}[theorem]{Lemma}   
\newtheorem{proposition}[theorem]{Proposition}   
\newtheorem{remark}[theorem]{Remark}   
  
\newtheorem{definition}{Definition}

\begin{document}
\title[Richards flow in porous media with cross-diffusion]{Richards flow in porous media with cross diffusion}

\author[E. S. Daus]{Esther S. Daus}
\address{Energie-Control Austria f\"ur die Regulierung der Elektrizit\"ats-und Erdgaswirtschaft (E-Control), Rudolfsplatz 13a, 1010 Vienna, Austria
}
\email{esther.daus@gmail.com}

\author[J.-P. Mili\v si\'c]{Josipa-Pina Mili\v si\'c}
\address{University of Zagreb, Faculty of Electrical Engineering and Computing, Unska 3, 10000 Zagreb, Croatia}
\email{pina.milisic@fer.unizg.hr}

\author[N. Zamponi]{Nicola Zamponi}
\address{Universit\"{a}t Augsburg, Institut f\"ur Mathematik, Universit\"{a}tsstra\ss e 12a, 86159 Augsburg, Germany}
\email{nicola.zamponi@uni-a.de}
\date{\today}

\date{\today}
\thanks{J.-P. Mili\v si\'c acknowledges the support by Croatian Science Foundation under the project IP-2025-02-5004 and the Republic of Croatia’s MSEY in course of Multilateral scientific and technological cooperation in Danube region under the project MultiHeFlo. N. Zamponi thanks the University of Augsburg for financial support via the research program “Forschungspotentiale besser nutzen!”. The results presented in this paper are those of the authors alone and do not represent the official views, policies, or positions of E-Control where the author E.S. Daus is employed. The authors bear full responsibility for the content of this work.}
   
\begin{abstract} 
We study a non-isothermal simplified two-phase air–water flow in porous media, where the water phase is modelled as a mixture of $N$ chemical components. The governing system consists of mass conservation equations for each component, an energy balance equation, and a capillary pressure relation. The model is thermodynamically consistent and incorporates cross-diffusion effects arising from multicomponent interactions.
Our main result establishes the sequential stability of weak variational entropy solutions. The analysis relies on a priori estimates derived from the entropy balance and the total energy balance, together with a dynamic capillary pressure law. Compactness arguments are carried out using the Div–Curl lemma, which allows us to pass to the limit in the nonlinear terms.

\end{abstract}



\maketitle


\section{Introduction} \label{sec.introl}

The modeling of multiphase flow through porous media plays a central role in many engineering applications, including geothermal systems, oil reservoir engineering, groundwater hydrology, and thermal energy storage. The specific motivation for our work arises from problems of groundwater contamination, where pollutants such as nitrates, arsenic, bacteria, and industrial chemicals infiltrate the upper layers of the soil.
In unsaturated soils, this situation can be described as a two-phase flow of immiscible fluids - air and water - where the air phase is typically assumed to remain at constant pressure. The resulting dynamics may be modeled by a non-isothermal Richards-type equation, in which the water phase is treated as a mixture of $N$ chemical components subject to cross-diffusion effects.

The model is derived in a thermodynamically consistent framework, starting from the Helmholtz free energy, and is governed by coupled equations for mass and energy balances together with appropriate constitutive relations. Mathematically, this leads to a quasilinear parabolic system characterized by several degeneracies but endowed with an entropy structure. The main analytical challenges stem from the nonlinear and degenerate character of the equations, the strong coupling between different components, and the presence of cross-diffusion terms.
   
The mathematical analysis of multiphase flows in porous media has its origins in applications to geological sciences and reservoir simulations. One of the first results on the existence of weak solutions for a simplified system describing non-isothermal two-phase flow in porous media was obtained by Bocharov and Monakhov \cite{BoMo88.1, BoMo88.2}. Later, in \cite{LiSun10}, the authors established the global existence of weak solutions for a non-isothermal, one-dimensional multicomponent heat–air–vapor transport model in porous textile materials. Existence results for Richards’ model, arising in the study of heat and moisture flow through partially saturated porous media, were obtained in \cite{BePaz17}. More recently, the existence of weak solutions for a non-isothermal, immiscible, compressible two-phase flow in heterogeneous porous media was proved in \cite{AJPP25}. 

The mathematical analysis of cross-diffusion systems has received considerable attention over the last two decades, particularly through the development of entropy methods (see, e.g., \cite{BSW12, Jue16}). In recent years, there has been growing interest in modeling cross-diffusion effects in thermodynamically consistent ways, especially for fluid mixtures (see, e.g., \cite{BH15, BPZ17, JMZ18, DDGG20, FHKM20, HJ21}). Seminal contributions on the thermodynamic theory of fluid mixtures can be found in \cite{BoDre15, Gio99}.

In essence, cross-diffusion occurs when the concentration gradient of one species induces a flux of another species. Mathematically, this leads to strongly coupled parabolic systems of reaction–diffusion type, in which the diffusion matrix is generally non-diagonal, non-symmetric, and not positive definite. Classical maximum principle arguments are typically unavailable for such systems. However, the presence of a formal  entropy structure allows for rigorous analysis. A central technique in entropy methods is the introduction of entropy variables, which often correspond physically to the chemical potentials of the system, providing both mathematical tractability and physical interpretability.

 Building on these developments, the present work focuses on the weak sequential stability of weak variational entropy solutions for a multicomponent  immiscible, compressible two-phase unsaturated flow in porous media with cross-diffusion effects, extending previous analyses \cite{DMZ20} to a fully non-isothermal and thermodynamically consistent setting.  
 To the best of our knowledge, this remains a relatively new topic: cross-diffusion phenomena in porous media have received only limited attention in the mathematical literature (see \cite{OstrRho20, LiSun10}).
Our proof strategy combines entropy methods (as in \cite{Jue16}) with techniques from fluid dynamics (cf. \cite{Feir04}). Similar approaches have been successfully applied in the work of Buli\v cek, J\"ungel, Pokorn\'y, and Zamponi \cite{BJPZ20}, who demonstrated the existence of weak solutions for a thermodynamically consistent, steady compressible Navier–Stokes–Fourier system describing chemically reacting mixtures.
 These results contribute to the rigorous mathematical understanding of multicomponent two-phase flows, providing a framework that is relevant both for the analysis of porous media processes and for the development of reliable numerical schemes in applications ranging from fluid dynamics to environmental modeling.
  
The paper is organized as follows. In Section\ref{sec.phys}, we introduce the key physical quantities, present the governing equations of the model, and state the constitutive relations together with the main mathematical assumptions. Section\ref{sec.sol} is devoted to the notion of weak solutions, referred to as weak variational entropy solutions, and contains the main result on weak sequential stability for smooth solutions. In Section\ref{sec.apriori}, we derive the necessary a priori estimates, while Section\ref{sec.pass} addresses the passage to the limit. Finally, an appendix collects several auxiliary technical results used throughout the analysis.
\subsection{The physical quantities and main assumptions}\label{sec.phys}
We consider a nonisothermal Richards flow process with cross-diffusion effects in a porous reservoir $\Omega \subset \R^3$, which is a bounded Lipshitz domain, with a nondeformable skeleton. The time interval of interest is $(0,\mathcal{T})$ and $\mathcal{Q} = \Omega \times (0,\mathcal{T})$.
The indexes $w,n,s$ correspond to the wetting (water) phase, the non-wetting (air) phase, and the skeleton. 
The Richards model we are considering is a simplified two-phase flow in a porous medium in which the air pressure is assumed to be constant. It is applicable in the upper layers of the soil with a well connected pore network. Since air is much more mobile than water and is in constant contact with the atmosphere, it can react almost immediately to any change in pressure, resulting in an almost constant air pressure. Since the air phase is assumed to be homogeneous and there is no mass transfer between the phases, it can be eliminated from the system.\\
\noindent{\sf The porosity function:} 
By $\Phi$ we denote the porosity of the domain $\Omega$, i.e. the volume of the pore space in unit volume of a continuum.
We assume that the porosity $\Phi : \Omega\to [0,1]$ is a Lebesgue-measurable function such that
\begin{equation}
\label{Phi.pos}
\mbox{ess}\inf_{x \in \Omega}\Phi(x) > 0,\qquad 
\mbox{ess}\sup_{x \in \Omega}\Phi(x) < 1.
                 \end{equation}
{\sf The temperature:} $T = T(x,t)$ stands for the temperature in $\Omega$. We assume that the medium is locally in a heat equilibrium state, i.e. $T_w = T_n = T_s = T$.\\
{\sf The phase pressures:} Standardly,  $p_w $ denotes the water pressure while $p_n = p_{at}$ represents the constant air-pressure. Throughout this article, for the water pressure we take  
\begin{equation}
p = p(x,t) = p_w - p_{at}.
\label{press_w}
\end{equation}
{\sf The water saturation.} By $S_w = S(x,t)$ we denote the water saturation in $\Omega$. 
Water saturation refers to the fraction of the pore space in a porous medium that is filled with water. \\
{\sf The capillary pressure:}  The constitutive equation relating the capillary pressure $P_c$ to the water saturation $S$ is classically given as an algebraic relationship between $P_c$ and $S$, i.e. $P_c(S) = p_n - p_w$. A detailed discussion of this relationship can be found e.g. in \cite{Bear68}. The relationship between $P_c$ and $S$ has been generalized on the basis of thermodynamical arguments by Gray and Hassanizadeh \cite{HassGrey93} (see also  \cite{CaoPop16}).

 In the case of Richards flow, using relation \eqref{press_w}, the dynamic capillary pressure relation can be written in the form  
 \begin{equation}
 p = -P_c(S) - \partial_t f(S),   
 \label{DCP}
 \end{equation}
 where  $f$ denotes the dynamic capillary pressure function.\\
{\sf The water mass density:} $\rho_w = \rho(T,p)$. We assume that the water phase is a mixture of $N$ different components. 
We denote the mass concentration of the $i-$th water component by $\rho_i$, for $i=1,\ldots,N$. More precisely $\rho_i$ is the mass of the $i$-th component in the volume of the water mixture. It holds that $\rho = \sum_{i=1}^N \rho_i$. Moreover, let be $\vec\rho = (\rho_1,\ldots,\rho_N)$.\\
 {\sf The energy densities of the water-phase, the skeleton and the interface between two phases.} By $E_w$ we denote the energy density of the water phase, neglecting the viscous dissipation and kinetic energy.
 We assume that the skeleton energy density $E_s$ depends on the temperature only, i.e. $E_s = E_s(T)$. 
   Following \cite{JKKP19} the total energy density of the fluid $E_{f}$ is given by
\begin{align}
\label{def.Ef}
E_{f} = E_w S  + E_{int}(S).
\end{align}
Here, the total water energy $E_f$ is given by the water internal energy multiplied by the water saturation $S$, plus the interfacial energy contribution; scaling by $S$ reflects the fraction of pore volume occupied by water, while the interface term $E_{int} = E_{int}(S)$ accounts for energy at phase boundaries \cite{JKKP19}.

Following \cite{Coussy04, JKKP19} the capillary pressure function is related to $E_{int}$ by
      \[  P_c(S) = - \frac{\partial E_{int}}{\partial S}.\]
  It follows that, up to a constant,  one has the following relation
  \begin{align}
\label{def.Eint}
E_{int}(S) &= \int_{S}^1 P_c(\xi)d\xi.
\end{align}
The total energy density $E_{tot}$ of the porous medium is given by
\[ E_{tot} =   \Phi E_f + (1-\Phi) E_s.     \]
{\sf The relative permeability function:} By $k_{r,w} = k_r(S)$ we denote the relative permeability of the water-phase in $\Omega$. We assume that $k_r$ is a continuous function that satisfies: $(i)$ $0 \leq k_r \leq 1$ on $\R$; $(ii)$ $k_r(S) = 0$ for $S \leq 0$ and $k_r(S) = 1$ for $S \geq 1$.\\
{\sf The mobility functions:}  The water mobility function $\lambda_w = \lambda(S,T)$ is defined by
        \begin{equation}
	\label{mobility}
\lambda(S,T) = \frac{k_{r}(S)}{\nu(T)},
\end{equation}
where $\nu(T)$ is the viscosity of the water phase.\\
%

\subsection{Model equations}\label{subsec.model}

  Let $\veb{v}$ be the water velocity and $\veb{v}_i$ be the velocity of the $i-$th water component, for $i=1,\ldots,N$.
A water mixture (barycentric) velocity $\veb{v}$ is defined as 
\[ \rho \veb{v} = \sum_{i=1}^N \rho_i \veb{v}_i. \] 
Following \cite{SDQ11}, each component of the water mixture satisfies the mass conservation law, i.e. 
\begin{equation}\label{eq.comp}
\Phi\frac{\partial}{\partial t} (S \rho_i) + \dv(  \rho_i \veb{v}_i ) = {r_i},\quad i=1,\ldots,N,
\end{equation}
where $r_1,\ldots,r_N$ are reaction terms modeling e.g.~chemical reactions between the components in the water mixture. We assume that the total mass conservation holds, i.e. 
\begin{align}
\sum_{i=1}^N r_i = 0.
\label{MassCons}
\end{align}
By summing equations \eqref{eq.comp} we get the total mass conservation of the water-phase:
\begin{equation*}
\Phi\frac{\partial}{\partial t} (S \rho) + \dv(  \rho\veb{v} ) = 0.
\end{equation*}
 Following \cite{Bear68}, the generalized Darcy law gives
\begin{align}
	\label{vow}
\veb{v} = -K(x)\lambda(S,T)\nabla p,
\end{align}
where $\lambda$ is the mobility function, and the absolute permeability of the domain $\Omega$, $K(x)$ is bounded, strictly positive function.
The transport of each water-component is divided into convective part ($\rho_i \veb{v}$)  and diffusive part
($\veb{J}_i$) by dividing the component flux    
\begin{equation*}
  \rho_i \veb{v}_i = \rho_i \veb{v} + \veb{J}_i,
\end{equation*}
where the diffusive flux is defined as $\veb{J}_i=\rho_i (\veb{v}_i - \veb{v})$. Obviously, 
$  \sum_{i=1}^N \veb{J}_i = 0.  $

Provided that the external forces are neglected, in the following we give 
expressions for the diffusion and the thermal fluxes, \cite{GrootMazur}. \\ 
{\sf  The diffusion flux  of the $i$th species} is denoted by  $\veb{J}_i$ and is given by
        \begin{align}
  \veb{J}_i = L_{i0} \nabla \Big(\frac{1}{T}\Big) - \sum_{j=1}^{N} L_{ij} \nabla \Big(\frac{\mu_j}{T}\Big), \qquad i=1,\ldots,N,  \label{flux.J}
\end{align}
where $L_{ij} = L_{ij}(\vec{\rho},T)$ are the diffusion coefficients (mobilities) which form the mobility matrix $(L_{ij})_{i,j=1,\ldots,N}$ and
$\mu_1,\ldots,\mu_N$ represent the chemical potentials of the components of the water-phase.
Furthermore, 
$L_{i0} = L_{i0}(\vec\rho,T)$ for $i=1,\ldots,N$.\\
{\sf The heat flux} is denoted by $\veb{q}$ and consists of Fourier's law and the molecular diffusion term
         \begin{align}
  \veb{q}  = -\kappa(T) \nabla T + \sum_{j=1}^{N} L_{0{j}} \nabla \Big( \frac{\mu_j}{T} \Big),
  \label{flux.q}
\end{align}
where $\kappa(T) > 0$ is the heat conductivity given by
\begin{equation}
 \kappa(T) = L_{00} T^{-2}.  \label{heat_cond}
\end{equation}  
More details of the contemporary approach of deriving appropriate 
representations of the fluxes can be found in \cite{BoDre15, BD20} and the references 
given there.

Following \cite{Kav95}, for a control volume in porous medium the energy conservation reads
\[  \frac{\partial}{\partial t} E_{tot}  +  \dv( E_{flux}  )   = 0, \]
where the energy flux $E_{flux}$ consists of the convective part $(E_w + p) \veb{v} $ and the diffusive part $\veb{q}$ related to the heat flux.
In this way one obtains the energy equation
\begin{align}
\frac{\partial}{\partial t}\big( \Phi (S E_w + E_{int}(S)) + (1-\Phi) E_s \big) 
+ \dv \big( (E_w + p) \veb{v}  + \veb{q} \big)  = 0.  
\label{eq.energy}
\end{align}

\begin{remark}
 We note that the energy conservation equation in a porous medium does not include the term of internal energy production. In porous media the friction between the fluid and the porous matrix dominates internal friction within the fluid which is therefore neglected. In contrast, compression effects are not completely neglected, and therefore we have the specific enthalpy $E_w + p$ under the divergence, instead of a simple specific internal energy. Production of internal energy by chemical reactions is also neglected.
\end{remark}

 Complete model consists of the mass conservation equation \eqref{eq.comp} for each water-component phase $i$, the energy conservation equation \eqref{eq.energy} and the capillary pressure law \eqref{DCP}:
\begin{align}
\Phi  \frac{\partial}{\partial t} (S \rho_i) + \dv(  \rho_i\veb{v} + \veb{J}_i ) &= {r_i},\quad i=1,\ldots,N,  \label{cons_w} \\
  \frac{\partial}{\partial t}\big (
\Phi (S E_w + E_{int}(S)) + (1-\Phi) E_s \big) 
+ \dv \big( (E_w + p) \veb{v}  + \veb{q} \big) & = 0,   \label{cons_e} \\
\partial_t f(S) + P_c(S) + p &= 0. \label{PC.Sw}
\end{align}

We assume the following initial conditions:
\begin{align}
	\label{ic.1}
	&\rho_i(\cdot,0) = \rho_{i,in},\quad  i=1,\ldots,N, \quad
	S(\cdot,0) = S_{in},\quad
	T(\cdot,0) = T_{in}\quad\mbox{in }\Omega,
\end{align}
where the initial data are Lebesgue measurable functions such that 
 equations \eqref{cons_w}-\eqref{PC.Sw} are solved in a bounded domain $\Omega \subset \R^3$ and are supplemented with complete slip boundary conditions for the velocity, and Robin boundary conditions for the diffusion fluxes and the heat flux on $\partial \Omega \times (0,\infty)$, $i=1,\ldots,N$:
\begin{align}\label{bc}
&\veb v \cdot \bn = 0,\quad
 \veb J_i\cdot\bn = \sum_{k=1}^Nb_{ik} 
 \left( \frac{\mu_k}{T} - \frac{\mu_{0,k}}{T_0} \right),\quad 
\veb q \cdot \bn = \alpha (T-T_0),
\end{align}
where $\alpha>0$, $T_0>0$, $\mu_{0,1},\ldots,\mu_{0,N}\in\R$ 
are scalar constants, and $\bn$ is a unit outer normal on $\Omega$.
We note that
$(b_{ij})_{i,j=1,\ldots N}\in\R^{N\times N}$ is a constant, symmetric, positive semidefinite matrix such that
\begin{align}
\label{hp.b}
\sum_{i=1}^N b_{ij} = 0\qquad j=1,\ldots,N.
\end{align}
\subsection{Constitutive relations} \label{subsec.constrel}

 We assume that the thermodynamic quantities appearing in \eqref{cons_w}--\eqref{PC.Sw}, i.e. the chemical potentials $\mu_i$, the pressure $p$ and the thermal energy densities of the water-phase $E_w = (\rho e)_w$ and the skeleton $E_s = (\rho e)_s$,
are induced by the Helmholtz free energy density $\rho\Psi$ in a thermodynamically consistent way. We take that the free energy density $\rho\Psi$ decomposes into a sum of water- and skeleton-related contributions
\begin{align*}
\rho \Psi(\vec\rho, T)  = (\rho \Psi)_w(\vec\rho, T) + (\rho\Psi)_s(T). 
\end{align*}
The system is complemented by a set of constitutive relations \cite{Gio99, BoDre15}, which read
\begin{align}
(\rho \Psi)_{\mathcal{I}} - T \frac{\partial (\rho \Psi)_{\mathcal{I}}}{\partial T}  = (\rho e)_{\mathcal{I}}, & \quad  \mathcal{I} = \{ w, s\} \label{CR.1}\\
 \frac{\partial(\rho \Psi)_w}{\partial \rho_i} &=  \mu_i \label{CR.2} \\
 \frac{\partial(\rho \Psi)_w}{\partial T}  =  -(\rho \eta)_w,  & \quad   \frac{\partial(\rho \Psi)_s}{\partial T} =  -(\rho \eta)_s \label{CR.3},
 \end{align}
where by $(\rho \eta)_{\alpha}$, $\alpha \in \{w,s\}$ we denoted the entropy density, along with standard Gibbs-Duhem relation \begin{equation}
(\rho \Psi)_w + p = \sum_{i=1}^{N} \rho_i \mu_i.
\label{Gibbs-Duhem}
\end{equation}
for determining the pressure $p$ from the free energy density $(\rho \Psi)_w$.

 In this work we specify the functions $(\rho \Psi)_w, (\rho\Psi)_s$ as follows:
\begin{align}\label{Helm.w}
(\rho\Psi)_w(\vec\rho,T) &= T\sum_{i=1}^N\rho_i\log\rho_i + \rho^\gamma - c_w\rho T\log T  + p_{at},\\
\label{Helm.s}
(\rho\Psi)_s(T) &= T -c_s T\log T.
\end{align}
Here $c_w$, $c_s>0$ are (scaled) heat capacities (of water and skeleton, respectively), $p_{at}>0$ is the atmospheric pressure
shift of the free energy. We point out that we take
\begin{equation}
\label{gamma}
\gamma > 2.
\end{equation}

\begin{remark}
The choice for the Helmholtz free energy \eqref{Helm.w} is explained in \cite[Remark 1.2]{BJPZ20}.
We note that the term  $T\sum_{i=1}^N\rho_i\log\rho_i$ represents the entropy of mixing for an ideal mixture of $N$ components. Next, the term $\rho T\log T$ comes from the thermal entropy contribution in compressible non-isothermal fluids, see for example \cite{DDGG20}. Finally, the term $\rho^\gamma$ is needed for mathematical analysis in order to obtain an estimate for the total mass density. A certain physical justification for this term can be found in \cite{FeirNov09}. The Helmholtz free energy $(\rho \Psi)_w$, in essentially the form like \eqref{Helm.w} can be found for example in \cite{DrGisKr14} for use in non-isothermal, thermodynamically consistent phase-field models. 
Furthermore, the choice \eqref{Helm.s} for the skeleton free energy 
is mathematically constructed to ensure thermodynamic consistency with the energy of the skeleton given by \eqref{Skel.Ener}.
 \end{remark}
  
As a consequence of \eqref{CR.1}-\eqref{CR.3} the thermodynamic quantities are given by the following relations.
To maintain simplicity, we will use the notation $(\rho e)$ and $(\rho \eta)$ instead of $(\rho e)_w$ and $(\rho \eta)_w$ in the following.
\smallskip\\
{\sf Chemical potentials:}
 \begin{align}
 \mu_i (\rho,T) =  \gamma\rho^{\gamma-1} + T(\log\rho_i + 1)  - c_w T\log T,
 \quad i=1,\ldots,N. \label{chp.1}
\end{align}
{\sf Pressure:}
\begin{align}
  p(\rho,T)  =  (\gamma-1)\rho^\gamma  + T\rho - p_{at}. \label{thp.1}
\end{align}
{\sf Water energy density: } 
\begin{align}
 E_w(\rho,T)  =  (\rho e)(\rho,T)  =  \rho^\gamma + c_wT \rho + p_{at}.
 \label{eng.1}
\end{align}
{\sf Water entropy density:}
\begin{align}\label{entropy}
(\rho \eta)(\rho,T) =  
-\sum_{i=1}^N\rho_i\log\rho_i + c_w\rho(\log T + 1).
\end{align}
{\sf Skeleton entropy and energy:}
\begin{align}
(\rho \eta)_s(T) &= 
c_s - 1 + c_s\log T,
\label{Skel.Entr}\\
E_s(T) &
=  (\rho e)_s (T) = c_s T. \label{Skel.Ener}
\end{align}
\begin{remark}\label{Est_Add.0-basic}
We point out that the following relations hold
\begin{align}
	\label{GibbsDuhem_Old}
	T d(\rho\eta) = d(\rho e) - \sum_{i=1}^N\mu_i d\rho_i ,
\end{align}
where the differential $d$ stands here for $\partial_t$ or $\partial_{x_i}$, $i=1,2,3$.
Note that  \eqref{GibbsDuhem_Old} gives:
\begin{equation}
\label{GD-1}
 \frac{\pa(\rho\eta)}{\pa\rho_i} =  -\frac{\mu_i}{T}, \quad
  \frac{\pa(\rho\eta)}{\pa (\rho e)} = \frac{1}{T}. 
\end{equation}
We also point out relation that connects all thermodynamic quantities
\begin{align}
		\label{rhoeta}
	p + (\rho e) - \sum_{i=1}^N\rho_i\mu_i = T(\rho\eta).
\end{align}
\end{remark}
\begin{remark}\label{Est_Add.1-basic}
	The following estimates hold:
		\begin{align*}
			(\rho e) &\leq C_1(p+ 2p_{at})\\
			p &\leq C_2 (\rho e), \quad p \geq -p_{at}\\
			E_{f} - S(\rho\eta)  &\geq S \rho^\gamma  - C_3,\\
			E_{s} - (\rho\eta)_{s}  &\geq C_4(T + |\log T|) - C_5,
		\end{align*}
		where $C_i$, $i=1,\ldots,5$ are  fixed constants.
	\end{remark}
\subsection{Hypotheses} \label{subsec.hypoth}
In this subsection we collect the main mathematical 
hypotheses which we imposed throughout this article, 
and introduce six important parameters:
$\gamma, \beta, q$, $a$, $\alpha_r$ and $k_p$.

{\bf H0.} {\sf The porous medium properties:}
The porosity function $\Phi$ satisfy (1) and absolute permeability function $K(x)$ 
satisfy $0 < K_{min} \leq K \leq K_{max} < +\infty$.

{\bf H1.} {\sf Initial data:}
\begin{align}
	\label{hp.data}
\rho_{i,in}>0,\quad  {0 < S_{in} < 1}, 
\quad T_{in}>0,\quad\mbox{a.e.~in }\Omega,\\
\label{hp.data.2}
{f(S_{in})}, \;  |\nabla f(S_{in})|^q, \;  S_{in}\rho_{in}, \;
\Phi E_{f,in} + (1-\Phi)E_{s,in},\;
\Phi S_{in} (\rho\eta)_{in} + (1-\Phi)(\rho\eta)_{s,in}\in L^1(\Omega),
\end{align}
where $q$ is defined in {\bf H3}.

{\bf H2.} {\sf Phase viscosity, relative permeability.} Function $\mu : \R_+\to\R_+$ is a continuous, uniformly positive and bounded, while  $k_r\in C^0([0,1])$ is an increasing, nonnegative function such that
\begin{equation}
\label{hp.kr}
\exists \alpha_r > \frac{2}{\gamma},\;
\exists k_r^*>0 \textrm{ such that }
\lim_{s\to 0} s^{-\alpha_r}k_r(s) = k_r^* ,
\end{equation}
where $\gamma$ is given in \eqref{Helm.w} and \eqref{gamma}.

{\bf H3.} {\sf Capillary pressure.} Functions $P_c , f : (0,1)\to\R$ are $C^1$ strictly decreasing, and there exists $q \in [\gamma/(\gamma-1),2 )$ and $c_f > 0$ such that
\begin{align}
\label{hp.Pcf}
\inf_{0<s<1/2}\frac{|P_c'(s)| \, k_r(s)^{\frac{q}{2(q-1)}}}{
	|f'(s)|} \geq c_{f},\\
\label{hp.krf}
\lim_{s\to 1}f(s) = -\infty,\quad
\exists c_f'>0 \textrm{ such that }
\left|\frac{d}{ds}\sqrt{k_r(s)}\right|\leq c_f'
\left|f'(s)\right|,\quad 0<s<1,\\
\label{hp.Pcf2}
f(0)>0,\quad P_c(S)>0,\quad
\inf_{s\in (0,s_0)}\frac{P_c(s)}{f(s)} > \frac{p_{at}}{f(0)},\;
\textrm{ where }\; f(s_0) = 0,\\
\label{hp.Pc.bound}
\exists c_{p} > 0,~~ k_p> 0 \textrm{ such that }
\lim_{s\to 0^+} s^{k_p} P_c(s) = c_p.
\end{align}
We note that conditions \eqref{hp.kr}, \eqref{hp.Pcf}, \eqref{hp.krf} and \eqref{hp.Pc.bound} imply a lower bound on $k_p$, equal to $\frac{\alpha_r}{2(q-1)}$. An example of functions $f$, $P_c$ and $k_r$ satisfying {\bf H2}  and {\bf H3}
can be given by :
\begin{equation*}
f(s) = 2- \frac{1}{(1-s)^a},\quad P_c(s) = \frac{1}{s^{k_p}},\quad 
k_r(s) = s^{\alpha_r},
\end{equation*}
where $\alpha_r > 2$, $p_{at} =1$, $a>0$, and 
$k_p \geq  \frac{q\alpha_r}{2(q-1)} -1$. We note also that the function $f(s)$
can be chosen up to an additive constant but condition \eqref{hp.Pcf2} 
demands that one can lift the curve $f$ such that $f(0) >0$ and \eqref{hp.Pcf2}
is still satisfied. This condition relates the shapes of the two curves to $p_{at}$. 

{\bf H4.} {\sf Heat conductivity.} Function $\kappa(T)$ satisfies the assumption that there exists an exponent $\beta > 0$ such that
\begin{align}
	\label{kappa}
\exists \kappa_1, \kappa_2>0 \textrm{ such that }
\kappa_1(1+T^\beta)\leq\kappa(T)\leq \kappa_2(1+T^\beta),
\end{align}
where 
\begin{align}
\label{hp.beta}
\beta \geq 
\max \left\lbrace  \frac{q}{2-q}, \frac{\gamma}{\gamma-2}, 2 \right\rbrace ,
\quad \beta \neq \frac{\gamma}{\gamma-2}, \quad \beta \neq 2.
\end{align}

{\bf H5.} {\sf Mobilities.} 
The Onsager-Casimir 
reciprocity relations imply the symmetry of the mobilities $L_{ij}=L_{ji}$, $L_{i0}=L_{0i}$ for $i,j=1,\ldots,N$,
and the second law of thermodynamics requires 
the positive semidefinitness of the matrix 
$(L_{ij}) \in \R^{N\times N}$ as well as 
$L_{00}\geq 0$. 
Furthermore, for all $i,j=1,\ldots,N$ the coefficients 
$L_{ij}, L_{i0}\in C^0(\R^N_+\times\R_+)$ satisfy
\begin{align}
\label{hp.Ltilde}
| L_{ij}(\vec\rho,T)| + \frac{1}{T}|L_{i0}(\vec\rho,T)|
\leq C,
\end{align}
for all $(\vec\rho,T) \in \R_+^N \times \R_+$. 
Moreover, we assume that there exists $C, C'>0$ such that
\begin{align}
		\label{Ass.L}
		C |\Pi \vec{u} |^2\leq 
	\sum_{i,j=1}^N L_{ij}(\vec\rho,T) u_i u_j \leq 
	C' |\Pi \vec{u} |^2 , \quad 
	\forall \vec{u}\in\R^N, \; \vec\rho\in\R_+^N, \; T>0,
\end{align}  
where $\Pi = \Ib - \veb{1} \otimes \veb{1}/N$, is the orthogonal projector on $\sspan \{\veb{1}\}^\top$.
Thus,
\begin{align}
	\label{Pi}
(\Pi \vec{u})_i = u_i - \frac{1}{N} \sum_{j=1}^N u_j,\quad i=1,\ldots,N.
\end{align}
We also point out that, being $(L_{ij})_{i,j=1,\ldots,N}$ symmetric and positive semidefinite, \eqref{Ass.L} implies that 
$\sum_{i=1}^N L_{ij}=0$ for $j=1,\ldots,N$. 
Moreover, we assume that
\begin{equation} \label{Li0.Sum}
\sum_{i=1}^N L_{i0}=0.
\end{equation}
Assumptions \eqref{hp.Ltilde}-\eqref{Ass.L} are standard in non-equilibrium thermodynamics: the mobility coefficients $L_{ij}$ scale naturally with the mass densities $\rho_i$, satisfy symmetry and conservation, and ensure boundedness and coercivity on the projected space.

{\bf H6.} {\sf Reaction terms.} 
Following  \cite{DDGG20} and we take 
in the following form
 \begin{align}
r_i = -\sum_{j=1}^l \frac{\partial \Psi}{\partial X_j}(D^R) \gamma_i^j, \mbox{ where } D_j^R = \sum_{k=1}^N
\gamma_k^j \frac{\mu_k}{T},\quad i=1,\ldots,N,\quad j=1,\ldots,l,
\label{React_1}
\end{align}
where $\Psi \colon \R^l \to \R$ is a convex potential with suitable growth and $\vec\gamma^j \in \R^N$ is a vectorial stoichiometric coefficient associated with $j$th reaction, while $l$ is a number of active reactions. 
We assume that there exists $C_1, C_2>0$ and $a>2$ such that
\begin{align}
 \Psi(X) \geq C_1 |X|^a, &\quad |\nabla \Psi(X)| \leq C_2 |X|^{a-1},
 \label{Add_Psi} \\
F(X):=X\cdot\nabla\Psi(X) & \textrm{ is convex, lower semicontinuous. }
\label{Add_Fatou.5}
\end{align}
From \eqref{MassCons} it follows that $\sum_{i=1}^N \gamma_i^j = 0$, $\forall j=1,\ldots,l$.
As a consequence, $\vec{\gamma}^j \in \sspan\{\veb{1}\}^{\perp}$.
Additionally, we assume that the linear hull of all $\veb{\gamma}^j$ is equal to $\sspan\{\veb{1}\}^{\perp}$.
 For $i=1,\ldots,N$ we have that $r_i = r_i(\Pi(\vec\mu/T))$. Moreover, from \eqref{React_1} and $\eqref{Add_Psi}_1$
 it follows that $\exists C_r >0$   such that 
 \begin{align}
        \label{def.rtilde}
	\displaystyle  -\sum_{i=1}^N r_i(\Pi(\vec \zeta))\zeta_i 
	\geq  C_r\left|\Pi\vec\zeta\right|^a.
\end{align}
\begin{remark}
	\label{remark.fluxes}
From the structure of the fluxes we have two useful relations. 
	First, from \eqref{flux.J}, \eqref{flux.q} and \eqref{heat_cond} it follows (see \eqref{Ass.L})
\begin{align}
	-\sum_{i=1}^N  \veb{J}_i\cdot \nabla \Big( \frac{\mu_i}{T} \Big)   + \veb{q}\cdot \nabla \Big( \frac{1}{T} \Big) 
	= \sum_{i,j=1}^N L_{ij} \nabla \big( \frac{\mu_i}{T} \big) \cdot \nabla \big( \frac{\mu_j}{T} \big) + \kappa(T) \big| \nabla \log T \big|^2 \geq 0.
	\label{Eq.9}
\end{align}
Second, using \eqref{chp.1}, \eqref{thp.1} and \eqref{eng.1} after some calculations we get
\begin{align}
	-\sum_{i=1}^N \rho_i \veb v \cdot \nabla \Big( \frac{\mu_i}{T} \Big)  
	+ ( (\rho e) + p ) \veb v \cdot \nabla \big( \frac{1}{T} \big)
	 = \frac{K}{T} \lambda(S,T) | \nabla p|^2. 
	 \label{Eq.9-2}
\end{align}
We note also that 
\begin{equation}
	\big(\Pi(\vec\mu/T)\big)_i = \log \rho_i - \frac{1}{N} \sum_{k=1}^N \log \rho_k.
	\label{Proj_1}
	\end{equation}
\end{remark}

\subsection{Entropy balance equation} \label{subsec.entropybal} 
 One of the components of the weak variational entropy solution concept for our problem is the entropy balance equation, which is derived from the mass and energy balance equations \eqref{cons_w} and \eqref{cons_e}. More precisely, one starts by multiplying \eqref{cons_w} by $-\mu_i/T$ and summing $i=1,\ldots,N$, and multiplies \eqref{cons_e} by $1/T$ to finally sum the obtained equations and consider the constitutive relations.
For a triplet $(\vec\rho,T,S): \Omega\times\R_+
\to \R^N_+\times\R_+\times [0,1]$ which is a smooth solution to  \eqref{cons_w}--\eqref{bc},
a straightforward manipulation yields
{\sf the entropy balance equation}
\begin{equation}\label{EBE}
\begin{split}
&\frac{\partial}{\partial t}\left[\Phi S(\rho\eta) + (1-\Phi)(\rho \eta)_s \right] + \dv\Big( (\rho \eta) \veb{v} - \sum_{i=1}^N \frac{\mu_i}{T} \veb{J}_i + \frac{\veb{q}}{T} \Big)   \\
& = \frac{K}{T}\lambda(S,T)|\nabla p|^2 +
  \sum_{i,j=1}^N L_{ij} \nabla \big( \frac{\mu_i}{T} \big) \cdot \nabla \big( \frac{\mu_j}{T} \big) 
  + \kappa(T) \big| \nabla \log T \big|^2 
- \Phi \frac{1}{T} f'(S) \big( \partial_t S \big)^2
-\sum_{i=1}^N r_i\frac{\mu_i}{T}. 
\end{split}
\end{equation}

The entropy production given by the right-hand side of \eqref{EBE} is nonnegative
thanks to the assumptions that  $(L_{ij}) \in \R^{N \times N}$ is a positive semidefinite matrix, $\kappa(T)\geq 0$, the mobility function $\lambda$ in \eqref{mobility} is nonnegative, $f'(S)\leq 0$ and $-\sum_{i=1}^N r_i\frac{\mu_i}{T}\geq 0$. 
Therefore, the second law of thermodynamics holds true.

\section{Weak formulation and main result} \label{sec.sol} 

In this section, we introduce the concept of a weak solution to our problem and formulate the main result of the paper. Specifically, we work with weak variational entropy solutions, a notion first introduced by Feireisl in \cite{Feir04} in the context of the Navier–Stokes–Fourier system. A central ingredient in our analysis is the proper formulation of the energy balance, which serves as the cornerstone of the weak sequential stability result. Following the approach of \cite[Chapter 2]{Feir04}, we combine the weak formulation of the entropy inequality with the integrated form of the total energy balance, thereby obtaining a framework suitable for the analysis of our two-phase multicomponent flow model.
\begin{definition}  \label{def.weaksol}
 We say that the Lebesgue-measurable function 
\[ (\vec\rho,T,S): \Omega\times\R_+
	\to \R^N_+\times\R_+\times [0,1], \]
is a weak variational entropy solution to \eqref{cons_w}--\eqref{bc} provided the following identities hold:
\begin{enumerate}
 \item[1.] The weak formulation of the partial mass balances  \eqref{cons_w}: for $i=1,\ldots,N$
   \begin{equation}
		\begin{split}
		& -\int_0^\tf\int_\Omega \Phi S \rho_i \partial_t \varphi dx dt
		- \int_0^\tf\int_\Omega(\rho_i \veb{v} + \veb{J}_i)\cdot\nabla\varphi dx dt
		 +\int_0^\tf\int_{\pa\Omega}\sum_{k=1}^Nb_{ik}\left(\frac{\mu_k}{T} - \frac{\mu_{0,k}}{T_0}\right)\varphi d\sigma dt \\
		 & \qquad = 
		\int_\Omega \Phi S_{in} \rho_{i,in} \varphi(\cdot,0) dx + \int_0^\tf\int_\Omega r_i\varphi dx dt, 
		\quad  \forall \varphi\in C^1(\overline{Q}_\tf), \; \varphi(\cdot, \tf) = 0;
		\end{split}
			\label{weak.1}
	\end{equation}
 \item[2.]  The integrated form of the total energy balance \eqref{cons_e}
\begin{equation}\label{weak.2}
\begin{split}
	\int_\Omega\big (\Phi E_f(t) + (1-\Phi) E_s(t) \big)dx
	& + \alpha\int_0^{t}\int_{\pa\Omega} (T-T_0)d\sigma dt \\
	&= 
	\int_\Omega\big (\Phi E_{f, in} + (1-\Phi) E_{s,in} \big)dx,
	\quad\forall t \in [0,\tf], 
	\end{split}
\end{equation}
where $ E_{f, in} = E_f(S_{in}, \rho_{in}, T_{in})$ and $E_{s,in} = E_s(T_{in})$.
 \item[3.]  The weak formulation of the entropy balance equation \eqref{EBE}
\begin{equation}\label{weak.EBE}
\begin{split}
        & \int_0^\tf\int_\Omega \Big[ \Phi S (\rho \eta)_w + (1-\Phi)(\rho \eta)_s \Big] \partial_t \varphi dx dt
         	 -\int_0^\tf\int_\Omega\Big( (\rho \eta) \veb{v} - \sum_{i=1}^N \frac{\mu_i}{T} \veb{J}_i + \frac{\veb{q}}{T} \Big)\cdot\nabla\varphi dx dt  \\
		& \qquad =\int_0^\tf\int_\Omega\left( 
		\sum_{i,j=1}^N L_{ij} \nabla \big( \frac{\mu_i}{T} \big) \cdot \nabla \big( \frac{\mu_j}{T} \big) 
		+ \kappa(T) \big| \nabla \log T \big|^2  \right) \varphi dx dt \\ 
		&\qquad + \int_0^\tf {\int_\Omega} \left( 
		\frac{K}{T}\lambda(S,T)|\nabla p|^2
		-\Phi \frac{1}{T} f'(S) \big( \partial_t S \big)^2
		-\sum_{i=1}^N r_i\frac{\mu_i}{T}
		\right)  \varphi dx dt \\
		&\qquad +\int_0^{\tf}\int_{\pa\Omega}\left(\alpha\frac{T_0 - T}{T}
		+\sum_{i,j=1}^N b_{ij}\frac{\mu_i}{T}\left(\frac{\mu_j}{T} - \frac{\mu_{0,j}}{T_0}\right)
		\right)\varphi d\sigma dt  \\
		& \qquad + \int_{\Omega}  \Big[ \Phi S_{in} (\rho \eta)_{in} + (1-\Phi) (\rho \eta)_{s,in} \Big] \varphi(\cdot,0) dx   + \langle\xi , \varphi\rangle,
	\end{split}
	\end{equation}
	for every $\varphi\in C^1(\overline{Q}_\tf)$, $\varphi(\cdot, \tf) = 0$, $\varphi\geq 0$ a.e.~in $Q_\tf$, for a suitable choice of $\xi\in \mathcal M(\overline{Q}_\tf)$ a nonnegative Radon measure;
 \item[4.] The dynamic capillary pressure relation \eqref{PC.Sw} holds a.e.~in $Q_\tf$. 	
  \end{enumerate}	
	 Furthermore, the constitutive relations  \eqref{thp.1}--\eqref{entropy}, \eqref{Skel.Entr}, \eqref{Skel.Ener}  are satisfied a.e.~$x\in\Omega$, $t>0$ 
for some suitable exponent
$a > 2$ and $\beta$, $q$ given by \eqref{hp.beta}, \eqref{hp.Pcf} (respectively),	
while the constitutive relation for 
the chemical potentials \eqref{chp.1} and the definition of 
the reaction terms \eqref{def.rtilde} are fulfilled in the non-vacuum 
region $Q_\tf\backslash\{\rho=0\}$. 
\end{definition}

In Definition \ref{def.weaksol}  the standard weak formulation of the energy equation is replaced by the weak formulation of the entropy equation and the global total energy balance. This modification is necessary because of the lack of compactness in the nonlinear flux term $(\rho e + p)\veb{v}$, a well-known difficulty in the analysis of time-dependent compressible multicomponent mixtures \cite{EZ15, BPZ17}. It is important to note, however, that in this setting
the entropy production rate is interpreted as a non-negative Radon measure, which ensures consistency with the thermodynamic framework.
Finally, in Definition \ref{def.weaksol}, we impose the constitutive relation for the 
chemical potentials \eqref{chp.1} and the definition of the reaction terms \eqref{def.rtilde} only in the non-vacuum region $Q_\tf\backslash\{\rho=0\}$.  This restriction is necessary because, in the vacuum region 
$\{\rho =0\}$, we cannot establish almost everywhere convergence of the projected relative chemical potentials due to the lack of compactness; see Remark \ref{VacuumRemark}. Consequently, the validity of the aforementioned relations in 
the vacuum region cannot be ensured within our framework. 

The main result of this paper is the following theorem.

\begin{theorem}[Weak sequential stability]\label{thm.ws}
Let $(\rhonv,\Tn,\Sn)$ be a sequence of smooth solutions to \eqref{cons_w}--\eqref{PC.Sw}, \eqref{bc} with initial data
\[
  \rho_i^{(n)}(\cdot,0) = \rho^{(n)}_{i,in}, \quad 
  S^{(n)}(\cdot,0) =  S_{in}^{(n)}, \quad 
  T^{(n)}(\cdot,0) = T_{in}^{(n)} .
\]
Assume further that $\rhon_i, \Tn, \Sn > 0$ a.e.~in $Q_\tf$ for $i=1,\ldots,N$ and $n\in\N$. 
Suppose that the following convergences of the initial data hold:
\begin{align*}
  S^{(n)}_{in}\rho^{(n)}_{i,in} &\to  S_{in}\rho_{i,in}, \\ 
  \Phi S^{(n)}_{in} (\rho\eta)^{(n)}_{in} + (1-\Phi)(\rho\eta)^{(n)}_{s,in} 
  &\to \Phi S_{in} (\rho\eta)_{in} + (1-\Phi)(\rho\eta)_{s,in},
\end{align*}
weakly in $L^1(\Omega)$ and
\[ \int_{\Omega} (\Phi E^{(n)}_{f,in} + (1-\Phi)E^{(n)}_{s,in} ) dx \to \int_{\Omega} (\Phi E_{f,in} + (1-\Phi)E_{s,in})dx. \]
Then, up to a subsequence, $(\rhonv,\Tn,\Sn)$ converges strongly in $L^1(Q_\tf)$ as $n\to\infty$ to a triplet $(\rhov,T,S)$, which is a variational entropy solution to \eqref{cons_w}--\eqref{bc} in the sense of Definition \ref{def.weaksol}.
\end{theorem}

The concept of weak sequential stability, in the sense of Feireisl \cite{Feir04}, together with weak compactness results, represents a crucial step in the analysis of global existence. We say that the system \eqref{cons_w}--\eqref{PC.Sw}, \eqref{ic.1}, \eqref{bc} is weakly sequentially stable if every sequence of solutions satisfying the a priori estimates established in Lemmas \ref{lem.tech.1}, \ref{Lemma_Gamma}, \ref{lem.ap.est} and Propositions \ref{prop.5}--\ref{prop.7} admits a subsequence that converges weakly to a variational entropy solution of the system \eqref{cons_w}--\eqref{PC.Sw}, \eqref{ic.1}, \eqref{bc}.

\section{A priori estimates}\label{sec.apriori}

In this section we give appropriate a priori estimates for a sequence $(\rhonv,\Tn,\Sn)$ of smooth solutions to to \eqref{cons_w}--\eqref{PC.Sw}, \eqref{ic.1}, \eqref{bc}. The estimates given in this section are valid for each $n \in \N$ but we omit writing the index.
 We start with estimates which depend only on the structure of equation \ref{PC.Sw}, given by Lemma \ref{lem.tech.1} and Lemma \ref{Lemma_Gamma}.
 
\begin{lemma}
\label{lem.tech.1}
 For smooth solution to \eqref{cons_w}--\eqref{PC.Sw}, \eqref{ic.1}--\eqref{bc} there exists a constant $C$  such that
 \begin{align} 
    & \|P_c(S)\|_{L^1(Q_\tf)} \leq C, \label{PC.Sw.Est1}\\
    &  \|f(S)\|_{L^{\infty}(0,\tf;L^{1}(\Omega))} \leq
    C(1+ \|Sp\|_{L^{\infty}(0,\tf;L^{1}(\Omega))}) \label{PC.Sw.Est1.New}\\
     & \|p\|_{L^1(Q_\tf)}  \leq C( 1+ \|Sp\|_{L^1(Q_\tf)}).  \label{PC.Sw.Est2}
 \end{align}
\end{lemma}

\begin{proof}
Let $\zeta : [0,1]\to [0,1]$ be a $C^1([0,1])$ 
cutoff function such that 
$\zeta = 1$ on $[0,1/3]$, $\zeta = 0$ on $[2/3,1]$, and
$\zeta$ is decreasing in $(1/3,2/3)$.
Let us multiply \eqref{PC.Sw} by $\zeta(S)$ and define
$f_1(s) = \int_0^s\zeta(u)f'(u)du$ for $0\leq s\leq 1$.
Notice that $f_1$ is bounded by {{\bf H3}} and the assumptions on $\zeta$. We obtain
\begin{align*}
\pa_t f_1(S) + \zeta(S)P_c(S) + \zeta(S)p = 0.
\end{align*}
Integrating the above identity leads to
\begin{align*}
\int_\Omega f_1(S(t))dx + \int_0^t\int_\Omega(\zeta(S)P_c(S) + \zeta(S)p)dx dt' = 
\int_\Omega f_1(S^{in})dx,\quad t>0.
\end{align*}
From \eqref{thp.1} we have that  $p\geq -p_{at}$ a.e.~in $Q_\tf$. On the other side, $P_c$ is a nonnegative function while $f_1$ is bounded, we obtain
that $\zeta(S)P_c(S)$ is bounded in $L^1(Q_\tf)$. Since 
$(1-\zeta(S))P_c$ is bounded (because $P_c$ is nonnegative, decreasing, and $1-\zeta$ vanishes near $0$), we infer  \eqref{PC.Sw.Est1}:
\begin{align}\label{int.PcS.ub}
\int_0^t\int_\Omega P_c(S)dx dt\leq C.
\end{align}
On the other hand, multiplying \eqref{PC.Sw} times $1-\zeta(S)$ and defining 
$f_2(s) = \int_0^s(1-\zeta(u))f'(u)du$ yields
\begin{align*}
\pa_t f_2(S) + (1-\zeta(S))P_c(S) + (1-\zeta(S))p = 0.
\end{align*}
Integrating the above identity leads to
\begin{align*}
-\int_\Omega f_2(S(t))dx =
\int_0^t\int_\Omega((1-\zeta(S))P_c(S) + 
(1-\zeta(S))p)dx dt' - 
\int_\Omega f_2(S^{in})dx,\quad t>0.
\end{align*}
We know that $(1-\zeta(S))P_c$ is bounded.
 Using estimate $(1-\zeta(S))|p| \leq 3S|p|$ {and {\bf H1}}, it follows that
\begin{align*}
-\int_\Omega f_2(S(t))dx\leq C (1+ \int_{\Omega} S|p|dx),\quad t>0.
\end{align*}
Since $-f_2(s) = -f_2(2/3)+f(2/3) - f(s)$ for $s>2/3$ and the fact that $f$ is bounded in $[0,2/3]$ (because it is decreasing, smooth in $(0,1)$, and $f(0)<\infty$), we conclude
\begin{align}\label{int.fS.lb}
-\int_\Omega f(S(t))dx\leq C (1+ \int_{\Omega} S|p|dx),\quad t>0,
\end{align}
so that \eqref{PC.Sw.Est1.New} follows.
Integrating \eqref{PC.Sw} and employing \eqref{int.PcS.ub}, \eqref{int.fS.lb} 
yields  \eqref{PC.Sw.Est2}.
\end{proof}
 
\begin{lemma}\label{Lemma_Gamma}
For any $\eps>0$ sufficiently small, there exists $\delta > 0$, depending only on $\eps$ such that
\[ S(t) \geq \eps \; \textrm{ for } \; t\geq \delta \; \textrm{ a.e. in }\; \Omega,      \]
for any smooth solution to \eqref{cons_w}--\eqref{PC.Sw}, \eqref{ic.1}--\eqref{bc}.
Consequently $S >0$ a.e. on  $Q_\tf$.
    \end{lemma}
\begin{proof}
According to \eqref{hp.Pcf2} we can select $f(0)> p_{at}/\Gamma$, where
\[\Gamma\equiv \inf_{0<s<s_0}\frac{P_c(s)}{f(s)}, \; \mbox{ with } \; f(s_0) = 0.\] 
 Since $p\geq -p_{at}$
and \eqref{PC.Sw} holds, it follows
\begin{align*}
\pa_t f(S) + \Gamma f(S) \leq p_{at}.
\end{align*}
A Gronwall argument yields, for $t >0$
\begin{align*}
f(S(t))\leq \Big( f(S^{in})   - \frac{p_{at}}{\Gamma} \Big)e^{-\Gamma t} + \frac{p_{at}}{\Gamma}, \; \mbox{ a.e.~in } \; \Omega .
\end{align*}
Monotonicity of the function $f$ together with the condition $f(0)> p_{at}/\Gamma$
gives that
for any $t \geq \delta >0$ we have
\begin{align*}
f(S(t))\leq f(0)e^{-{\Gamma} \delta} 
+ \frac{p_{at}}{{\Gamma}}(1-e^{-{\Gamma} \delta}),
\quad \textrm{ a.e.~in }\; \Omega.
\end{align*}
For given $0< \eps < p_{at}/\Gamma$ we can select sufficiently small $\delta >0$ 
such that 
\[ S(t) \geq  f^{-1}\Big(f(0)e^{-{\Gamma} \delta} 
+ \frac{p_{at}}{{\Gamma}}(1-e^{-{\Gamma} \delta})\Big) = \eps. \]
Consequently, we have $S >0$ a.e. on  $Q_\tf$.
\end{proof}
   The main a priori estimates resulting from the equations for the entropy and energy balance are given by Lemma \ref{lem.ap.est}.
The derivation of the a priori estimates is based on the application of the entropy equation in conjunction with constitutive relations and hypotheses. In the mathematical analysis, we used standard techniques:  Poincar\'e inequality (see \cite{FeirNov09}, Lemma 10.9 (ii)), Sobolev embeddings, and the interpolation arguments.
 
\begin{lemma}\label{lem.ap.est}
Any smooth solution to \eqref{cons_w}--\eqref{PC.Sw}, \eqref{ic.1}--\eqref{bc} satisfies
\begin{align}
	\|\Phi E_f + (1-\Phi) E_s\|_{L^\infty(0,\tf; L^1(\Omega))} &\leq C,
	\label{Est.E.LinfL1.2}\\
\|S^{1/\gamma}\rho\|_{L^\infty(0,\tf; L^\gamma(\Omega))} + 
\|S p\|_{L^\infty(0,\tf; L^1(\Omega))} &\leq C,\label{Est.E.LinfL1.1} \\
\|p\|_{L^1(Q_\tf)} +
\|\rho\|_{L^\gamma(Q_\tf)} &\leq C,\label{Est.p}\\
	{ \|\log T\|_{L^\infty(0,\tf; L^1(\Omega))} }+ 
 \|\log T \|_{L^2(0,\tf; H^1(\Omega))} 
 &\leq C, \label{Est.A}\\
 \label{Est.T}
	\|T\|_{L^\infty(0,\tf; L^1(\Omega))}  + \|T^{\beta/2} \|_{L^2(0,\tf; H^1(\Omega))}
	+ \|T\|_{L^{ \beta + \frac{2}{3} }(Q_\tf)} &\leq C,\\
 \|\Pi \big( \mub/T \big)\|_{L^2(0,\tf; H^1(\Omega))} 
 +\|\Pi \big( \mub/T \big)\|_{L^a(Q_\tf)} 
 &\leq C,
\label{Est.B.2}\\
\left\|\sqrt{\frac{\lambda(S,T)}{T}}\nabla p\right\|_{L^2(Q_\tf)}
&\leq C,\label{AJPP.1}\\
\label{mu.boundary} 
\|T^{-1}\|_{L^1(\pa\Omega\times (0,\tf))} + 
\left\| \sum_{i,j=1}^N b_{ij}\frac{\mu_i}{T}\frac{\mu_j}{T} \right\|_{L^1(\pa\Omega\times (0,\tf))} &\leq C,\\
 \|T^{-1} f'(S) \big( \partial_t S \big)^2\|_{L^1(Q_\tf)} +
\|f(S)\|_{L^\infty(0,\tf; W^{1,q}(\Omega))}  +{\|\pa_t f(S)\|_{L^1(Q_\tf)}}  
&\leq C, \label{Est.S}
\end{align}
where here and in the following, $C > 0$ denotes a generic constant depending only on the given data and (possibly) on $\tf$.
\end{lemma}
 \begin{proof}
We start by integrating the energy balance equation \eqref{cons_e} in space and time and using
the boundary conditions \eqref{bc}. It follows
\begin{align}\label{Energ}
\int_\Omega\big (
\Phi E_f + (1-\Phi) E_s \big)\vert_{t=0}^{t=t_1}dx
+ \alpha\int_0^{t_1}\int_{\pa\Omega} (T-T_0)
 d\sigma dt = 0,\quad\forall t_1 \in (0,\tf),
\end{align}
giving
\begin{equation}
\label{est.totE}
\sup_{t\in [0,\tf]}\int_\Omega\big (
\Phi E_f + (1-\Phi) E_s \big)dx
+ \alpha\int_0^\tf\int_{\pa\Omega} T d\sigma dt\leq C,
\end{equation}
from where \eqref{Est.E.LinfL1.2} directly follows.
Using $S \rho^\gamma \leq E_f $ (see   \eqref{def.Ef}, \eqref{eng.1}) and $Sp \leq C S (\rho e)$ 
(see Remark \ref{Est_Add.1-basic}),  
the $L^\infty(0,\tf; L^\gamma(\Omega))$ bound for $S^{1/\gamma}\rho$
and the $L^\infty(0,\tf; L^1(\Omega))$ bound for $Sp$ follow.
Therefore \eqref{Est.E.LinfL1.1} holds. 
From $E_s = c_s T$ using  \eqref{est.totE}  
it follows $\|T\|_{L^\infty(0,\tf; L^1(\Omega))}\leq C$.
Estimate  on pressure in \eqref{Est.p} follows from \eqref{Est.E.LinfL1.1} and \eqref{PC.Sw.Est2}.
The estimate on $\rho$  in  \eqref{Est.p} follows from  $(\gamma -1) \rho^\gamma \leq p + p_{at}$  (see \eqref{thp.1}).

Let us now integrate the entropy balance equation \eqref{EBE} with respect to space and time (from $t=0$ to $t=t_1$) and employ \eqref{ic.1}--\eqref{bc}. We get
 \begin{equation}
 \begin{split}
  \int_0^{t_1} \int_\Omega 
  &\Big( \sum_{i,j = 1}^N  L_{ij} \nabla \big(  \frac{\mu_i}{T}\big) \cdot \nabla \big(   \frac{\mu_j}{T} \big) + \kappa(T)|\nabla \log T|^2   \Big) dx dt  \\
  & + \int_0^{t_1}\int_\Omega
  \frac{K}{T}\lambda(S,T)|\nabla p|^2 dx dt
  -\int_0^{t_1}\int_\Omega\sum_{i=1}^N r_i\frac{\mu_i}{T}dx dt \\
  & +\int_0^{t_1}\int_{\pa\Omega}\left(\alpha\frac{T_0 - T}{T}
  +\sum_{i,j=1}^N b_{ij}\frac{\mu_i}{T}\left(\frac{\mu_j}{T} - \frac{\mu_{0,j}}{T_0}\right)
   \right)d\sigma dt\\
   & {-\int_0^{t_1}\int_\Omega \frac{\Phi}{T} f'(S) (\partial_t S)^2 dx dt} = \int_\Omega \Big(  \Phi S(\rho\eta)  
  + (1-\Phi)(\rho \eta)_s   \Big) dx\vert_{t=0}^{t=t_1},
  \end{split}
 \label{EBE.1}
 \end{equation}
 where $\kappa(T) $ is given by \eqref{heat_cond}.
 Summing \eqref{Energ} and \eqref{EBE.1} we get  the global entropy-energy equality for smooth solutions:
 \begin{equation}
 \label{GEE.Ineq}
 \begin{split}
     & \int_\Omega \Big(\Phi (E_f - S(\rho \eta)) + (1-\Phi)(E_s - (\rho \eta)_s \Big) dx\vert_{t=t_1} \\ 
  & +\int_0^{t_1}\int_{\Omega} \Big( \sum_{i,j = 1}^N  L_{ij} \nabla \big(  \frac{\mu_i}{T}\big) \cdot \nabla \big(   \frac{\mu_j}{T} \big) + \kappa(T)|\nabla \log T|^2   \Big) dx dt  \\
  & + \int_0^{t_1}\int_\Omega
  \frac{K}{T}\lambda(S,T)|\nabla p|^2 dx dt
  -\int_0^{t_1}\int_\Omega\sum_{i=1}^N r_i\frac{\mu_i}{T}dx dt \\
  & +\int_0^{t_1}\int_{\pa\Omega}\left(\alpha (T_0 - T) \left(\frac{1}{T} - 1 \right)
  +\sum_{i,j=1}^N b_{ij}\frac{\mu_i}{T}\left(\frac{\mu_j}{T} - \frac{\mu_{0,j}}{T_0}\right)
   \right)d\sigma dt\\
   & -\int_0^{t_1}\int_\Omega \frac{\Phi}{T} f'(S) (\partial_t S)^2 dx dt  = \int_\Omega \Big(  \Phi (E_f - S(\rho \eta)) + (1-\Phi)(E_s - (\rho \eta)_s  \Big) dx\vert_{t=t_0},
   \end{split}
 \end{equation}
Let us now consider the boundary term in \eqref{EBE.1}. 
Since the matrix $(b_{ij})_{i,j=1}^N$ is positive semidefinite and \eqref{hp.b} holds, 
via Cauchy-Schwartz inequality it follows 
\[
\sum_{i,j=1}^N b_{ij}\frac{\mu_i}{T}\left(\frac{\mu_j}{T} - \frac{\mu_{0,j}}{T_0}\right) \geq 
\frac{1}{2}\sum_{i,j=1}^N b_{ij}\frac{\mu_i}{T}\frac{\mu_j}{T} 
-\frac{1}{2}\sum_{i,j=1}^N b_{ij}\frac{\mu_{0,i}}{T_0}\frac{\mu_{0,j}}{T_0}.
\]
Furthermore, \eqref{def.rtilde} imply
\[
-\int_0^\tf\int_\Omega\sum_{i=1}^N r_i\frac{\mu_i}{T}dx dt
\geq C_r\|\Pi(\frac{\vec\mu}{T})\|^a_{L^a(Q_\tf)}. 
\]
On the other hand, \eqref{Ass.L} implies that 
\[
 \sum_{i,j=1}^N L_{ij} \nabla \big( \frac{\mu_i}{T} \big) \cdot
  \nabla \big( \frac{\mu_j}{T} \big)  
       \geq C |\nabla \Pi \big( \frac{\vec\mu}{T}  \big) |^2.
\]
From \eqref{kappa} it follows
\begin{align} 
	\kappa(T)|\nabla \log T|^2 & \geq \kappa_1(1+T^\beta) |\nabla \log T|^2 \geq c(|\nabla \log T|^2 + |\nabla T^{\beta/2}|^2 ).  \label{Eq.15}
\end{align}

In this way, taking into account above estimates and Remark~\ref{Est_Add.1-basic},
we obtain the final estimate: 
 \begin{equation}
 \label{GEE.Ineq-1}
 \begin{split}
     & C\int_\Omega \Big( S\rho^\gamma + T + |\log T|) \Big) dx\vert_{t=t_1} \\ 
  & + C\int_0^{t_1}\int_{\Omega}\Big(   |\nabla \Pi \big( \frac{\vec\mu}{T}  \big) |^2
	 + |\nabla \log T|^2   + |\nabla T^{\beta/2}|^2 \Big) dx dt  \\
  & + \int_0^{t_1}\int_\Omega
  \frac{K}{T}\lambda(S,T)|\nabla p|^2 dx dt
	+ C_r\|\Pi(\frac{\vec\mu}{T})\|^a_{L^a(Q_\tf)}  \\
	 & +\int_0^{t_1}\int_{\pa\Omega}\left(\alpha \left(T+\frac{T_0}{T}\right) + \frac{1}{2}\sum_{i,j=1}^N b_{ij}\frac{\mu_i}{T}\frac{\mu_j}{T} \right)d\sigma dt\\
   & +\int_0^{t_1}\int_\Omega \frac{\Phi}{T} | f'(S) | (\partial_t S)^2 dx dt \\
    & \leq \int_\Omega \Big(  \Phi (E_f - S(\rho \eta)) 
	+ (1-\Phi)(E_s - (\rho \eta)_s  \Big) dx\vert_{t=t_0} \\ 
	 & +\int_0^{t_1}\int_{\pa\Omega}\left(\alpha (1+T_0)  
	 + \frac{1}{2} \sum_{i,j=1}^N b_{ij}\frac{\mu_{0,i}}{T_0}\frac{\mu_{0,j}}{T_0}
	  \right)d\sigma dt
   \end{split}
 \end{equation}
Most of the bounds in (\ref{Est.A})--(\ref{Est.S}) follow directly from (\ref{GEE.Ineq-1}). Furthermore,
the $ L^2(0,\tf; H^1(\Omega))$-bound on $\log T$ in (\ref{Est.A}) follow form Poincar\'e inequality, 
the $ L^2(0,\tf; H^1(\Omega))$-bound on $T^{\beta/2}$ in (\ref{Est.T}) is a consequence of Lemma~\ref{lema-TT}.
Finally, the last estimate in (\ref{Est.T}) is a consequence of a simple interpolation between 
spaces $ L^\infty(0,\tf; L^1(\Omega))$ and  $ L^\beta(0,\tf; L^{3\beta}(\Omega))$.
 Bound \eqref{Est.r} on the reaction terms $r_1,\ldots,r_N$ comes straightforwardly from \eqref{def.rtilde}, \eqref{Est.B.2}.
 
 Let us now show \eqref{Est.S}. 
The bound for the first term in \eqref{Est.S} follows immediately from \eqref{GEE.Ineq-1}), so it remains to bound the second term.
Take the gradient of \eqref{PC.Sw}, and multipy the resulting equation times 
$|\nabla f(S)|^{q-2}\nabla f(S)$ and integrate in $\Omega\times [0,t_1]$, where $t_1\in [0,\tf]$ is generic. It follows
\begin{align*}
&\frac{1}{q}\int_\Omega |\nabla f(S)|^q dx\Big\vert_{t=0}^{t=t_1} 
+ \int_0^{t_1}\int_\Omega |f'(S)|^{q-2} f'(S)P_c'(S) |\nabla S|^q dx dt\\
&\qquad
= -\int_0^{t_1}\int_\Omega |\nabla f(S)|^{q-2}\nabla f(S)\cdot\nabla p dx dt .
\end{align*}
Young's inequality and \eqref{AJPP.1} lead to, for any $\varepsilon > 0$,
\begin{align*}
&\frac{1}{q}\int_\Omega |\nabla f(S)|^q dx\Big\vert_{t=0}^{t=t_1} 
+ \int_0^{t_1}\int_\Omega |f'(S)|^{q-2}f'(S)P_c'(S) |\nabla S|^q dx dt\\
&\quad\leq \frac{\eps}{2}\int_0^{t_1}\int_\Omega\frac{T}{\lambda(S,T)} 
|\nabla f(S)|^{2(q-1)} dx dt + 
\frac{1}{2\eps}
\int_0^{t_1}\int_\Omega\frac{\lambda(S,T)}{T}|\nabla p|^2 dx dt\\
&\quad\leq \frac{\eps (q-1)}{q}\int_0^{t_1}\int_\Omega\lambda^{-q/(2q-2)}(S,T) |\nabla f(S)|^q dx dt + 
\frac{\eps (2-q)}{2q}\int_0^{t_1}\int_\Omega T^{q/(2-q)} dx dt 
+ C\eps^{-1} .
\end{align*}
Thanks to \eqref{mobility} and the boundedness of $\nu$, it follows
$1/\lambda(S,T) = \nu(T)/k_r(S)\leq C/k_r(S)$. Furthermore \eqref{hp.beta}, 
\eqref{Est.T} yield that $\|T^{q/(2-q)}\|_{L^1(Q_\tf)}\leq C$.
From these estimates and the fact that both $P_c$ and $f$ are nonincreasing, we deduce
\begin{align*}
&\frac{1}{q}\int_\Omega |\nabla f(S)|^q dx\Big\vert_{t=0}^{t=t_1} 
+ \int_0^{t_1}\int_\Omega |f'(S)|^{q-1} |P_c'(S)| |\nabla S|^q dx dt\\
&\qquad\leq C\eps\int_0^{t_1}\int_\Omega
k_r(S)^{-q/(2q-2)}|f'(S)|^q |\nabla S|^q dx dt + C\eps^{-1} .
\end{align*}
Since $k_r$ is uniformly positive in $[1/2,1]$, it follows
\begin{align*}
&\frac{1}{q}\int_\Omega |\nabla f(S)|^q dx\Big\vert_{t=0}^{t=t_1} 
+ \int_0^{t_1}\int_\Omega |f'(S)|^{q-1} |P_c'(S)| |\nabla S|^q dx dt\\
&\quad\leq 
   C\eps   \int_0^{t_1} \int_\Omega  |\nabla f(S(t))|^q dx dt\\
  &\qquad + C\eps\int_{Q_{t_1}\cap\{ S<1/2 \} }
k_r(S)^{-q/(2q-2)}|f'(S)|^q |\nabla S|^q dx dt + C\eps^{-1} .
\end{align*}
Applying \eqref{hp.Pcf} with {$\eps< c_{f}/C$ }
leads to
\begin{align*}
\int_\Omega |\nabla f(S)|^q dx\Big\vert_{t=0}^{t=t_1} 
	\leq C + 
	{C \int_0^{t_1}\int_\Omega |\nabla f(S(t))|^q dxdt},\quad \forall t_1\in [0,\tf],
\end{align*}
which via a Gronwall argument  and {\bf H1} yields a bound for $\nabla f(S)$ in $L^\infty(0,\tf; L^q(\Omega))$. However, since \eqref{int.fS.lb} holds and $f$ is upper bounded, \eqref{Est.S} follows via Poincar\'e's Lemma.  

An estimate \eqref{Est.S} for $\pa_t f(S)$ is immediately found from \eqref{PC.Sw}, \eqref{PC.Sw.Est1} and
\eqref{Est.p}.
This finishes the proof of Lemma~\ref{lem.ap.est}.
\end{proof}

In the proof of Lemma \ref{lem.ap.est} we have used the following result which is proved in the Appendix.

\begin{lemma}
	\label{lema-TT}
Let $T \in L^{\infty}(0,\tf;L^1(\Omega))$, $ \nabla T^{\beta/2} \in L^2(Q_\tf)$ 
	and $\nabla \log T \in L^2(Q_\tf)$. 
	Then $T^{\beta/2} \in L^2(0,\tf; H^1(\Omega))$ 
	and  $\|T^{\beta/2}\|_{L^2(0,\tf; H^1(\Omega))} \leq C$ , where constant $C$ depends on the following norms: 
	$\|T\|_{L^\infty(0,\tf;L^1(\Omega))}$, $\|\nabla T^{\beta/2} \|_{L^2(Q_\tf)}$ and $\|\nabla \log T \|_{L^2(Q_\tf)}$.
\end{lemma}

The following three propositions (Prop. \ref{prop.5}, Prop. \ref{prop.6}, Prop. \ref{prop.7}) give  a priori estimates needed for the existence result.
Their proofs are rather technical and they are given in the Appendix.
The first one gives estimates for the terms in \eqref{cons_w}, \eqref{cons_e}.
\begin{proposition}\label{prop.5}
Any smooth solution to \eqref{cons_w}--\eqref{PC.Sw}, \eqref{ic.1}--\eqref{bc} satisfies
\begin{align}
& \|r_i\|_{L^{\frac{a}{a-1}}(Q_\tf)} \leq C,\quad i=1,\ldots,N, \label{Est.r} \\
& {\| \veb v \|_{L^{\frac{2\beta}{\beta+1}}(Q_\tf)} \leq C}, \label{Est.rhov}\\
& \|\veb J_i\|_{L^2(Q_\tf)}\leq C,\quad i=1,\ldots,N. \label{Est.J} \\
& \|\rho_i\veb v + \veb J_i\|_{L^m(Q_\tf)} \leq C,\quad i=1,\ldots,N. \label{Est.flux.1} \\
 & \|\pa_t(\Phi S\rho_i)\|_{L^{\widetilde{m}}(0,\tf; W^{-1,\widetilde{m}}(\Omega))} \leq C, \; \mbox{ for } \widetilde{m} = \min(m, \frac{a}{a-1}) > 1.
\label{Est.Srhot} \\
& \|\pa_t F(S)\|_{L^\infty(0,\tf; L^1(\Omega))}\leq C,\quad
F(s) \equiv -\int_s^{1/2} s_1 f'(s_1)ds_1 \geq 0, \label{Est.FSt}
\end{align}
where $\gamma$ and $\beta$ were given by \eqref{gamma}, \eqref{hp.beta} and
\[
 1 < {m\equiv \frac{2\beta\gamma}{\beta(2+\gamma) + \gamma} < 2}.
\]
\end{proposition}

The following proposition gives the bound of the total entropy $\Phi S(\rho\eta) + (1-\Phi)(\rho\eta)_s$ and for the entropy flux.
\begin{proposition}\label{prop.6}
Any smooth solution to \eqref{cons_w}--\eqref{PC.Sw}, \eqref{ic.1}--\eqref{bc} satisfies
\begin{align}\label{Est.kappaT}
	\left\|\kappa(T)\nabla\log T\right\|_{L^{\frac{2+3\beta}{1+3\beta}}(Q_\tf)}
	+
	\left\|\sum_{j=1}^N \frac{L_{0j}}{T}\nabla\frac{\mu_j}{T}\right\|_{L^2(Q_\tf)}\leq C,
\end{align}

\begin{equation}
\label{Est.entr}
\|\Phi S(\rho\eta) + (1-\Phi)(\rho\eta)_s\|_{L^{\frac{2\gamma}{\gamma+2}}(Q_\tf)} \leq C.
\end{equation}
\begin{equation}
\label{Est.entrflux}
\exists s>1:\quad 
\left\|(\rho\eta)\veb v -\sum_{i=1}^N\frac{\mu_i}{T}\veb J_i +
\frac{\veb q}{T}\right\|_{L^s(Q_\tf)} \leq C .
\end{equation}
\end{proposition}

Finally, we derive the gradient bound for $\rho$, the bound for $\log (\rho_i/\rho)$.
\begin{proposition}\label{prop.7}
Any smooth solution to \eqref{cons_w}--\eqref{PC.Sw}, \eqref{ic.1}--\eqref{bc} satisfies
\begin{equation}
\label{Est.narho}
\|\nabla [\sqrt{k_r(S)} G(\rho^\gamma)]\|_{L^{a_2}(0,\tf; L^q(\Omega))}\leq C[G],\qquad \forall G\in W^{1,\infty}(\R_+),
\end{equation}
for some $a_2 >1$.
\begin{equation}
\label{est.relchp}
\left\| \log\frac{\rho_i}{\rho} \right\|_{L^2(Q_\tf)}\leq C,\qquad
i=1,\ldots,N.
\end{equation}
\end{proposition}

\section{Proof of Theorem \ref{thm.ws}}\label{sec.pass}

In this section we assume that $(\rhonv,\Tn,\Sn)$ is a sequence of smooth solutions given in Theorem \ref{thm.ws}, and we show that $(\rhonv,\Tn,\Sn)$ converges (up to subsequences) to some variational entropy solution $(\rhov,T,S)$ of \eqref{cons_w}--\eqref{PC.Sw}, \eqref{ic.1}, \eqref{bc} as $n\to \infty$.

\subsection{Convergence results}
  In the following lemma we will collect the strong convergence results needed for passing to the limit in \eqref{weak.1}--\eqref{weak.2}--\eqref{weak.EBE}, when $n \to \infty$. From now on we will use the notation $Q_\tf^\eps = \Omega\times [\eps,\tf]$, where $\eps > 0$. 
  
  Let us denote with $\overline{u^{(n)}}$ a weak $L^1$-limit of the sequence $(u^{(n)})_{n \in \mathbb{N}}$, meaning that $\overline{u^{(n)}} \in L^1(\Omega)$ is defined by
$$ \int_{\Omega}\overline{u^{(n)}}\phi\,dx := \lim_{n\to\infty} \int_{\Omega} u^{(n)}\phi\,dx \quad \mbox{ for all } \; \phi \in L^{\infty}(\Omega). $$
  
 We recall a very useful compactness criterion from \cite[Corollary 10.2]{FeirNov09}. This will be used as the final step to upgrade almost everywhere convergence to strong convergence.
 \begin{proposition}\label{delaValle}
Let the sequence $(f_n)$ be bounded in $L^p(\Omega)$ with $1 < p < \infty$, $|\Omega| < \infty$, and suppose that  $f_n \to f$ almost everywhere in $\Omega$. Then $f_n \to f$ strongly in $L^q(\Omega)$ for all $q$ such that $1 \leq q < p$.
\end{proposition} 
 
In order to prove weak compactness, the Div-Curl lemma \cite[Prop. 3.3]{FeirNov09}, developed by Murat \cite{Murat78} and Tartar \cite{Tartar75}, has been used, which represents an efficient tool for handling compactness in nonlinear problems, where the classical Rellich--Kondraschev argument is not applicable. The Div-Curl lemma ensures that for certain sequences $(\Theta_n)$ with weak limits, nonlinear functional behave well. Roughly, 
\begin{align}
\overline{f(\Theta_n)} = f(\overline{\Theta}_n).
\label{DC-1}
\end{align}
where $f$ is strictly convex (or concave) functions. 
The identity \eqref{DC-1} is a powerful tool because it usually forces the weak limit to agree with the pointwise limit. More precisely, it implies that no oscillations or concentrations are left in the sequence $(\Theta_n)$, see \cite{FeirNov09}, Theorems 10.19, and 10.20.
Therefore, one can conclude 
\[ \Theta_n \to   \overline{\Theta}_n   \; \mbox{a.e. in } \Omega.  \]
Finally, combining this a.e. convergence with the boundedness of $(\Theta_n)$ in some $L^p(\Omega)$, with $ 1 < p < \infty$, Proposition \ref{delaValle} yields
\[ \Theta_n \to   \overline{\Theta}_n   \; \mbox{ strongly  in } L^q(\Omega),\; \forall 1 \leq q < p.
\]
    
  \begin{lemma}\label{str.convergences}
The following results hold:
\begin{equation}
\label{S.pos}
 S>0 \; \mbox{ a.e. in }\; Q_\tf.
\end{equation}
\begin{align}
\Sn &\to S \mbox{ strongly in } L^r(Q_\tf), \forall r < \infty,
\label{S.cnv}\\
\rhon &\to \rho \mbox{ strongly in } L^{\gamma-\eps}(Q_\tf), \label{rho.cnv}\\
\rhon_i &\to\rho_i \mbox{ strongly in } L^{\gamma-\eps}(Q_\tf),\quad
i=1,\ldots,N, \label{rhoi.cnv}
\end{align}
\begin{align}
	\label{cnv.aux}
\Sn\rhon_i\log\rhon_i\to S\rho_i\log\rho_i,\quad
\Sn\rhon_i\to S\rho_i,\quad i=1,\ldots,N,\\
\nonumber
\mbox{strongly in }L^{\gamma-\eps}(Q_\tf),
\end{align}
\begin{align}
	\label{T.cnv}
\log\Tn &\to\log T \mbox{ strongly in } L^{2-\eps}(0,\tf; L^{6-\eps}(\Omega)),\\
\label{T.cnv.2}
\Tn &\to T \mbox{ strongly in } L^{\beta + \frac{2}{3} - \eps}(Q_\tf).
\end{align}
\begin{align}
p^{(n)} & \to p \mbox{ strongly in } L^{4/3}(Q_\tf^\eps) \mbox{ for every } \eps > 0, \label{p.cnv.Qeps} \\
\Sn p^{(n)}& \to S p \mbox{ strongly in } L^1(Q_\tf). \label{p.cnv} 
\end{align}
\end{lemma}

\begin{proof}
This proof will be divided into 5 steps. In Step 1 we will prove \eqref{S.cnv}, \eqref{S.pos}.
 In Step 2 we will prove \eqref{rho.cnv}. Step 3 will be about proving \eqref{rhoi.cnv} and \eqref{cnv.aux}, and Step 4 proves \eqref{T.cnv} and \eqref{T.cnv.2}.   Finally, in Step 5 we prove \eqref{p.cnv.Qeps} and \eqref{p.cnv}.
 
 Let us define preliminarly the following vector fields:
\begin{equation}
	\label{Un.i}
U_i^{(n)} = (\Phi \Sn\rhon_i, \rhon_i\vn + \Jn_i),\qquad
i=1,\ldots,N,
\end{equation}
\begin{equation}
\label{Un}
U^{(n)} = \sum_{i=1}^N U^{(n)}_i = (\Phi \Sn\rhon, \rhon\vn),
\end{equation}
\begin{equation}
\label{Vn}
V^{(n)}[G] = (\sqrt{k_r(\Sn)}G( (\rhon)^{\gamma}), 0,0,0),\qquad
G\in W^{1,\infty}(\R_+),
\end{equation}
\begin{equation}
\label{Zn}
Z^{(n)}[G] = ( G(\Pi\big(\vec\mu^{(n)}/T^{(n)}\big)),0,0,0),\qquad
G\in W^{1,\infty}(\R^N),
\end{equation}
\begin{equation}
\label{Wn}
W^{(n)} = (\Phi\Sn(\rho\eta)^{(n)} + (1-\Phi)(\rho\eta)_s^{(n)},  (\rho\eta)^{(n)} \veb{v}^{(n)} - \sum_{i=1}^N \frac{\mu_i^{(n)}}{\Tn} \veb{J}_i^{(n)} + \frac{\veb{q}^{(n)}}{\Tn} ),
\end{equation}
\begin{equation}
\label{Yn}
Y^{(n)}[G] = ( G(\Tn),0,0,0),\qquad G\in W^{1,\infty}(\R_+).
\end{equation}

{\em \textbf{Step 1: Strong convergence and a.e. positivity of saturation.}}

In this step we show \eqref{S.cnv}.
From \eqref{Est.S} it follows
\begin{align}\label{bounds.fSn}
\|\pa_t f(\Sn)\|_{L^1(0,\tf; L^1(\Omega))} + 
\|f(\Sn)\|_{L^\infty(0,\tf; W^{1,q}(\Omega))}\leq C.
\end{align}
Since $W^{1,q}(\Omega)\hookrightarrow L^q(\Omega)$ compactly and 
$L^q(\Omega)\hookrightarrow L^1(\Omega)$ continuously, the Aubin-Lions lemma \cite{Simon87}
allows us to deduce that, up to subsequences,
\begin{align*}
f(\Sn)\quad\mbox{is strongly convergent in $L^q(Q_\tf)$ as $n\to \infty$.}
\end{align*}
In particular $f(\Sn)$ is a.e.~convergent in $Q_\tf$. Since $f$ is strictly decreasing (and a fortiori one-to-one), it follows that $\Sn$
is a.e.~convergent in $Q_\tf$. Thanks to the uniform $L^\infty(Q_\tf)$ bounds for $\Sn$, we deduce \eqref{S.cnv}.
Positivity of $S^{(n)}$ as well as positivity of the limit $S$ \eqref{S.pos} is direct consequence of Lemma \ref{Lemma_Gamma} and convergence a.e. in $\Qt$. 

\medskip
\noindent

{\em \textbf{Step 2: Strong convergence of the total density.}} 

In this step we show \eqref{rho.cnv}.
Since \eqref{Est.p} holds, it follows that (up to subsequences)
$$
\rhon\rightharpoonup\rho\quad\mbox{weakly in }L^\gamma(Q_\tf).
$$
Let us consider the vector fields $U^{(n)}$, $V^{(n)}[G]$ defined in \eqref{Un}, \eqref{Vn}, with $G\in W^{1,\infty}(\R_+)$ arbitrary. Thanks to \eqref{Est.E.LinfL1.1}, \eqref{Est.flux.1}, we deduce that $U^{(n)}$ is bounded in $L^{m}(Q_\tf)$ with $m>1$, while $V^{(n)}[G]$ is trivially bounded in $L^\infty(Q_\tf)$. On the other hand,
summing \eqref{cons_w} in $i=1,\ldots,N$ yields $\dv_{(t,x)}U^{(n)}=0$,
while the antisymmetric part of the Jacobian of $V^{(n)}[G]$, which we denote by $\curl_{(t,x)} V^{(n)}[G]=\nabla V^{(n)}[G]-\nabla^T V^{(n)}[G]$, satisfies
$$
|\curl_{(t,x)} V^{(n)}[G]|\leq C |\nabla [\sqrt{k_r(\Sn)}G( (\rhon)^{\gamma})]|
$$
and is therefore bounded in $L^1(Q_\tf)$ thanks to \eqref{Est.narho}.
In particular both $\dv_{(t,x)}U^{(n)}$, $\curl_{(t,x)} V^{(n)}[G]$ are relatively compact in $W^{-1,z}(Q_\tf)$ for any $z>1$. Therefore we can apply the Div-Curl Lemma \cite[Prop. 3.3]{FeirNov09} and deduce
\begin{align*}
\overline{U^{(n)}\cdot V^{(n)}[G]} = 
\overline{U^{(n)}}\cdot \overline{V^{(n)}[G]},
\end{align*}
Using \eqref{Un} and \eqref{Vn}, this is equivalent to
\begin{align*}
\overline{\Phi\Sn\rhon \sqrt{k_r(\Sn)}G( (\rhon)^{\gamma})} = 
\overline{\Phi\Sn\rhon}\,
\overline{\sqrt{k_r(\Sn)}G( (\rhon)^{\gamma})} .
\end{align*}
However, since $\Phi>0$ and does not depend on $n$, while $\Sn\to S$
strongly in $L^p(Q_\tf)$ for every $p<\infty$, it follows
\begin{align*}
	S\sqrt{k_r(S)}\,
	\overline{\rhon G( (\rhon)^{\gamma})} = 
	S\sqrt{k_r(S)}\,
	\rho\,\overline{G( (\rhon)^{\gamma})} .
\end{align*}
It holds $S\sqrt{k_r(S)}>0$ because $S>0$ a.e.~in $Q_\tf$.
We deduce
\begin{align}\label{rGr}
\overline{\rhon G( (\rhon)^{\gamma})} = 
\rho\,\overline{G( (\rhon)^{\gamma})} \quad
\mbox{on }Q_\tf , \quad
\forall G\in W^{1,\infty}(\R_+).
\end{align}
Define $G_\infty(s) = (1+s)^{1/\gamma}$.
Let us now choose $G(s) = G_k(s)\equiv \min\{ G_\infty(s), (1+k)^{1/\gamma}  \}$,
for $s\geq 0$, $k\in\N$. The weak lower semicontinuity of the $L^1$-norm and Cauchy-Schwartz inequality allow us to estimate
\begin{align*}
&\|
\overline{\rhon [G_k( (\rhon)^{\gamma})-G_\infty( (\rhon)^{\gamma})] }
\|_{L^1(Q_\tf)}\\ 
&\quad\leq\liminf_{n\to\infty}\|
\rhon [G_k( (\rhon)^{\gamma})-G_\infty( (\rhon)^{\gamma})] 
\|_{L^1(Q_\tf)}\\ 
&\quad\leq\sup_{n\in\N}\|\rhon\|_{L^2(Q_\tf)}
\|G_k( (\rhon)^{\gamma})-G_\infty( (\rhon)^{\gamma})\|_{L^2(Q_\tf)}.
\end{align*}
Thanks to \eqref{Est.p}, it follows (remember that $\gamma>2$):
\begin{align*}
&\|\overline{\rhon [G_k( (\rhon)^{\gamma})-G_\infty( (\rhon)^{\gamma})] }\|_{L^1(Q_\tf)}^2\\
&\quad\leq C\sup_{n\in\N}\int_{Q_\tf\cap\{\rhon > k^{1/\gamma}\}}
(1+(\rhon)^\gamma)^{2/\gamma}dx dt\\
&\quad\leq C k^{2/\gamma-1}\sup_{n\in\N}\int_{Q_\tf\cap\{\rhon > k^{1/\gamma}\}}
(\rhon)^{\gamma-2}(1+(\rhon)^\gamma)^{2/\gamma}dx dt\\
&\quad\leq C k^{2/\gamma-1}\sup_{n\in\N}\|\rhon\|_{L^\gamma(Q_\tf)}^\gamma
\leq C k^{2/\gamma-1} .
\end{align*}
In a similar way one shows that
\begin{align*}
\|\rho\,
\overline{G_k( (\rhon)^{\gamma})-G_\infty( (\rhon)^{\gamma}) }
\|_{L^1(Q_\tf)}^2\leq C k^{2/\gamma-1}.
\end{align*}
From the above inequalities and \eqref{rGr} with $G=G_k$ it follows
\begin{align*}
&\|\overline{\rhon G_\infty( (\rhon)^{\gamma})} - 
\rho\,\overline{G_\infty( (\rhon)^{\gamma})}\|_{L^1(Q_\tf)}\\
&\leq \|
\overline{\rhon [G_k( (\rhon)^{\gamma})-G_\infty( (\rhon)^{\gamma})] }
\|_{L^1(Q_\tf)} \\
&\qquad + \|\rho\,
\overline{G_k( (\rhon)^{\gamma})-G_\infty( (\rhon)^{\gamma}) }
\|_{L^1(Q_\tf)}\leq
C k^{1/\gamma-1/2}\to 0\quad\mbox{as }k\to\infty,
\end{align*}
implying that 
\begin{align}\label{id.rho}
	\overline{\rhon G_\infty((\rhon)^{\gamma})} = 
	\rho\,\overline{G_\infty((\rhon)^{\gamma})} \quad
	\mbox{on }Q_\tf .
\end{align}
Since $s\in\R_+\mapsto G_\infty(s^\gamma) = (1+s^\gamma)^{1/\gamma}\in\R_+$ is strictly increasing and strictly convex, we conclude \cite[Thr.~10.19, Thr.~10.20]{FeirNov09} that, up to subsequences,
$\rhon$ is a.e.~convergent in $Q_\tf$. Since $\rhon$ is bounded in $L^\gamma(Q_\tf)$, Proposition \ref{delaValle} implies that \eqref{rho.cnv} holds. 

\medskip
\noindent
{\em \textbf{Step 3: Strong convergence of partial densities.}}\\
In this step we show \eqref{rhoi.cnv} and \eqref{cnv.aux}.
For every $i=1,\ldots,N$,
the vector field $U_i^{(n)}$ defined in \eqref{Un.i} is bounded in 
$L^m(Q_\tf)$ thanks to \eqref{Est.flux.1}, while its time-space divergence $\dv_{(t,x)}U_i^{(n)} = r_i^{(n)}$ is bounded in $L^{1}(Q_\tf)$ due to \eqref{Est.r}, and a fortiori relatively compact in $W^{-1,p}(Q_\tf)$ for some $p>1$.
On the other hand, for every $G\in W^{1,\infty}(\R^N)$,
the vector field $Z^{(n)}[G]$ defined in \eqref{Zn} is trivially bounded in $L^\infty(Q_\tf)$,
while \eqref{Est.B.2} implies
\begin{align*}
\|\curl_{(t,x)}Z^{(n)}[G]\|_{L^2(Q_\tf)}\leq C[G]\|\nabla \Pi\big(\vec\mu^{(n)}/T^{(n)}\big) \|_{L^2(Q_\tf)}\leq C[G].
\end{align*}
In particular $\curl_{(t,x)}Z^{(n)}[G]$ is relatively compact in $H^{-1}(Q_\tf)$.
Therefore we can apply the Div-Curl Lemma with respect to the vector fields
$U_i^{(n)}$, $Z^{(n)}$ and obtain
\begin{align*}
\overline{U_i^{(n)}\cdot Z^{(n)}} = 
\overline{U_i^{(n)}}\cdot\overline{Z^{(n)}},\qquad i=1,\ldots,N,
\end{align*}
which means
\begin{align*}
\overline{\Phi\Sn\rhon_i G(\Pi\big(\vec\mu^{(n)}/T^{(n)}\big))} = 
\overline{\Phi\Sn\rhon_i}\,
\overline{G(\Pi\big(\vec\mu^{(n)}/T^{(n)} \big))},\qquad i=1,\ldots,N.
\end{align*}
after using \eqref{Un.i} and \eqref{Zn}.
Once again, since $\Phi > 0$ and does not depend on $n$, the strong convergence of $\Sn$ and the weak convergence of $\rhon_i$ imply
\begin{align}\label{riGmuT}
\overline{\rhon_i G(\Pi\big(\vec\mu^{(n)}/T^{(n)}\big))} = 
\rho_i\,
\overline{G(\Pi\big(\vec\mu^{(n)}/T^{(n)}\big))},\quad i=1,\ldots,N,\quad\mbox{on }Q_\tf ,
\end{align}
for every $G\in W^{1,\infty}(\R^N)$.
However, from {\eqref{Proj_1}}  it follows
\begin{equation}
\label{relchp}
(\Pi(\vec\mu^{(n)}/\Tn))_i = \log\rhon_i - \frac{1}{N}\sum_{j=1}^N\log\rhon_j ,\quad i=1,\ldots,N,
\end{equation}
so \eqref{riGmuT} becomes
\begin{align*}
	\overline{\rhon_i G\left(\log\rhon_i - \frac{1}{N}\sum_{j=1}^N\log\rhon_j\right)} = 
	\rho_i\,
	\overline{G\left(\log\rhon_i - \frac{1}{N}\sum_{j=1}^N\log\rhon_j\right)},\quad i=1,\ldots,N,\quad\mbox{on }Q_\tf .
\end{align*}
Note that from $\rhon\to\rho$ strongly in $L^1(Q_\tf)$ it follows that 
$\rhon \to \rho$  almost everywhere (up to a subsequence).
For an arbitrary $k\in\N$ let us define
$ Q_{\tf,k}\equiv\{ (x,t)\in Q_\tf ~:~ \rho(x,t)\geq 1/k \}$.
Now, for any $\epsilon>0$ and for any fixed $k \in \N$, since  $\rhon \to \rho$ 
a.e. and $\rho\geq 1/k$ in $Q_{\tf,k}$, Egorov's theorem ensures the existence of a measurable subset $ E_{\epsilon,k}\subset Q_{\tf,k}$ such that $\rhon\to\rho$  uniformly in $E_{\epsilon,k}$ and 
$ |Q_{\tf,k}\backslash E_{\epsilon,k}|<\epsilon$.
As a result $\rhon\to\rho$  strongly in  $L^\infty(E_{\epsilon,k})$.

We can therefore assume w.l.o.g.~that $\rhon\geq 1/2k$ a.e.~in $E_{\epsilon,k}$, $n\in\N$. Defining $\sigma_i^{(n)} = \rhon_i/\rhon$ on $E_{\epsilon,k}$, for $i=1,\ldots,N$, 
allows us to write
\begin{align*}
\sum_{i=1}^N\overline{\rhon_i G\left(\log\sigma^{(n)}_i - \frac{1}{N}\sum_{j=1}^N\log\sigma^{(n)}_j\right)} = 
\sum_{i=1}^N\rho_i\,
	\overline{G\left(\log\sigma^{(n)}_i - \frac{1}{N}\sum_{j=1}^N\log\sigma^{(n)}_j\right)}
	\quad\mbox{on }E_{\epsilon,k} .
\end{align*}

Fix $i\in\{1,\ldots,N\}$ generic. By choosing $G(\veb u) = G_M(u_i) = \min\{ u_i, M \}$ in the above identity, with $M\in\N$ generic, exploiting \eqref{Est.B.2} and proceeding like in the proof of \eqref{id.rho} one finds out that
\begin{align*}
\sum_{i=1}^N\overline{\rhon_i \left(
	\log\sgn_i - \frac{1}{N}\sum_{j=1}^N\log\sgn_j\right)} = 
\sum_{i=1}^N\rho_i\,
\overline{\left(
	\log\sgn_i - \frac{1}{N}\sum_{j=1}^N\log\sgn_j\right)}\quad
\mbox{on }E_{\epsilon,k},
\end{align*}
which, thanks to \eqref{rho.cnv} and \eqref{est.relchp}, leads to
\begin{align*}
\sum_{i=1}^N\overline{\rhon_i \log\sgn_i} = 
\sum_{i=1}^N\rho_i\,
\overline{\log\sgn_i}\quad \mbox{on }E_{\epsilon,k}.
\end{align*}
Once again, \eqref{rho.cnv} and \eqref{est.relchp} imply
\begin{align*}
	\sum_{i=1}^N\overline{\sgn_i \log\sgn_i} = 
	\sum_{i=1}^N\overline{\sgn_i}\,
	\overline{\log\sgn_i}\quad \mbox{on }E_{\epsilon,k} .
\end{align*}
The fact that $\log$ is strictly monotone and strictly concave allows us to conclude that (up to subsequences)
$\sgn_i = \rhon_i/\rhon$ is a.e.~convergent on $E_{\epsilon,k}$,
which, together with \eqref{rho.cnv}, implies that 
$\rhon_i\to \rho_i$ a.e.~in $E_{\epsilon,k}$.
Since $|Q_{\tf,k}\backslash E_{\epsilon,k}|<\epsilon$ and \eqref{Est.p} holds, we easily deduce that 
$\rhon_i\to \rho_i$ strongly in $L^1(Q_{\tf,k})$ for every $k\in\N$.
Since $0\leq\rho\leq 1/k$ on $Q_\tf\backslash Q_{\tf,k}$, it follows that 
\begin{align*}
\limsup_{n\to\infty}\int_{Q_\tf}|\rhon_i - \rho_i|dx dt &\leq
\limsup_{n\to\infty}\int_{Q_{\tf,k}}|\rhon_i - \rho_i|dx dt +
\limsup_{n\in\N}\int_{Q_\tf\backslash Q_{\tf,k}}|\rhon_i - \rho_i|dx dt\\
&\leq 2 \int_{Q_\tf\backslash Q_{\tf,k}}\rho\, dx dt\leq
\frac{2}{k}|Q_\tf|,
\end{align*}
and therefore $\rhon_i\to \rho_i$ strongly in $L^1(Q_{\tf})$,
for $i=1,\ldots,N$.

Together with \eqref{Est.p}, we conclude \eqref{rhoi.cnv}. Moreover, \eqref{cnv.aux} holds as it is a straightforward consequence of \eqref{S.cnv}, \eqref{rhoi.cnv} and Proposition \ref{delaValle}.

\medskip
\noindent
{\em \textbf{Step 4: Strong convergence of the temperature.}} \\
In this step we show \eqref{T.cnv} and \eqref{T.cnv.2}.
Let us consider the vector fields $W^{(n)}$, $Y^{(n)}$ defined in \eqref{Wn}, \eqref{Yn}. Thanks to \eqref{Est.entr}, \eqref{Est.entrflux}
we deduce that $W^{(n)}$ is bounded in $L^s(Q_\tf)$ for some $s>1$,
while $Y^{(n)}[G]$ is bounded in $L^\infty(Q_\tf)$ for every $G\in W^{1,\infty}(\R_+)$. On the other hand, the time-space divergence
$\dv_{(t,x)}W^{(n)}$ is the right-hand side of \eqref{EBE}, which is bounded in $L^1(Q_\tf)$ thanks to Lemma \ref{lem.ap.est}, while
$$
|\curl_{(t,x)}Y^{(n)}[G]|\leq C[G]|\nabla T^{{(n)}}|
$$
is bounded in 
$L^1(Q_\tf)$ due to \eqref{Est.A}. 

We observe that
\begin{align*}
\nabla \Tn = \chf{(0,1)}(\Tn) \Tn \nabla\log\Tn + \chf{[1,\infty)}(\Tn)\frac{2}{\beta}(\Tn)^{1-\beta/2}\nabla (\Tn)^{\beta/2}
\end{align*}
is bounded in $L^2(Q_T)$ given the uniform bounds for $\nabla\log\Tn$ and $\nabla (\Tn)^{\beta/2}$ in $L^2(Q_T)$ from Lemma \ref{lem.ap.est}.
By Proposition \ref{prop.6}  both 
$\dv_{(t,x)}W^{(n)}$ and $\curl_{(t,x)}Y^{(n)}[G]$ are relatively compact in $W^{-1,r}(Q_\tf)$ for some $r>1$. From the Div-Curl Lemma it follows
\begin{align*}
\overline{W^{(n)}\cdot Y^{(n)}[G]} = 
\overline{W^{(n)}}\cdot \overline{Y^{(n)}[G]}
\end{align*}
which means
\begin{align*}
\overline{(\Phi\Sn(\rho\eta)^{(n)} + (1-\Phi)(\rho\eta)_s^{(n)})G(\Tn)}
=
\overline{(\Phi\Sn(\rho\eta)^{(n)} + (1-\Phi)(\rho\eta)_s^{(n)})}
\overline{G(\Tn)}\\
\mbox{a.e~in }Q_\tf,\quad\forall G\in W^{1,\infty}(\R_+).
\end{align*}
From the definitions \eqref{entropy}, \eqref{Skel.Entr},
bound \eqref{Est.A} and strong convergence relations \eqref{cnv.aux} 
we deduce
\begin{align*}
(c_w\Phi S \rho + c_s(1-\Phi))\overline{(\log\Tn)G(\Tn)} = 
(c_w\Phi S \rho + c_s(1-\Phi))\overline{(\log\Tn)}~
\overline{G(\Tn)},
\end{align*}
which, thanks to \eqref{Phi.pos}, yields
\begin{equation*}
\overline{(\log\Tn)G(\Tn)} = \overline{(\log\Tn)}~ \overline{G(\Tn)}\quad
\mbox{a.e.~in }Q_\tf,\quad
\forall G\in W^{1,\infty}(\R_+).
\end{equation*}
Since \eqref{Est.A} holds, by arguing in a similar way as the derivation of 
\eqref{id.rho} one obtains
\begin{equation}
\label{TGT}
\overline{(\log\Tn)\Tn} = \overline{(\log\Tn)}~ \overline{\Tn}
\quad\mbox{a.e.~in }Q_\tf .
\end{equation}
Once again, the strict monotonicity and strict convexity of the function {$-\log$} 
allows us to conclude that $\Tn$ is (up to subsequences) a.e.~convergent in $Q_\tf$. From this fact, \eqref{Est.A}, \eqref{Est.T} and Sobolev's embedding $H^1(\Omega)\hookrightarrow L^6(\Omega)$ it follows
\eqref{T.cnv} and \eqref{T.cnv.2}.

\medskip
\noindent
{\em \textbf{Step 5: Strong convergence of the pressure.}} \\
Finally, we show \eqref{p.cnv.Qeps} and \eqref{p.cnv}.
Let us note that by \eqref{thp.1}, \eqref{rho.cnv}, \eqref{T.cnv.2} we have $p^{(n)}\to p$ a.e.~in $Q_\tf$.  We note that the strong convergence of the saturation \eqref{S.cnv} and its positivity \eqref{S.pos} will help us to  show \eqref{p.cnv}.
From \eqref{mobility}, \eqref{AJPP.1} and {Lemma \ref{Lemma_Gamma}}
as well as the uniform boundedness of the viscosity $\nu$ it follows
\begin{align*}
	\forall\eps>0,\quad\exists C_\eps>0:\quad
	\|(\Tn)^{-1/2}\nabla p^{(n)}\|_{L^2(Q_\tf^\eps)}\leq C_\eps .
\end{align*}
From the above estimate, \eqref{Est.T} and the fact that $\beta>2$ we obtain via H\uml older's inequality
\begin{align*}
	\forall\eps>0,\quad\exists C_\eps>0:\quad
	\|\nabla p^{(n)}\|_{L^{16/11}(Q_\tf^\eps)}\leq 
	\|(T^{(n)})^{1/2}\|_{L^{16/3}(Q_\tf^\eps)}
	\|(\Tn)^{-1/2}\nabla p^{(n)}\|_{L^{2}(Q_\tf^\eps)} \leq 
	C_\eps .
\end{align*}
Via Poincar\'e Lemma we get
\begin{align*}
	\forall\eps>0,\quad\exists C_\eps>0:\quad
	\left\|p^{(n)} - |\Omega|^{-1}\int_\Omega p^{(n)} dx\right\|_{L^{16/11}(Q_\tf^\eps)}\leq C_\eps .	
\end{align*}
On the other hand, \eqref{Est.E.LinfL1.1} and {Lemma \ref{Lemma_Gamma}} yield
\begin{align*}
	\left|\int_\Omega p^{(n)}(t) dx\right|\leq \frac{1}{s_\eps}
	\int_{\Omega}\Sn(t) |p^{(n)}(t)|dx\leq \frac{C}{s_\eps},\quad
	t\geq \eps,
\end{align*}
which means
\begin{align}\label{Est.p2.a}
	\forall\eps>0,\quad\exists C_\eps>0:\quad
	\|p^{(n)}\|_{L^{16/11}(Q_\tf^\eps)}\leq C_\eps .
\end{align}
Since $p^{(n)}\to p$ a.e.~in $Q_\tf$, 
Proposition \ref{delaValle} implies \eqref{p.cnv.Qeps}.
Given that \eqref{S.cnv}, \eqref{p.cnv.Qeps} hold, it follows
\begin{align*}
	\Sn p^{(n)}\to S p\quad\mbox{strongly in }L^{1}(Q_\tf^\eps),\quad\forall\eps>0.
\end{align*}
However, 
\begin{align*}
	\|\Sn p^{(n)}\|_{L^{1}(0,\eps; L^1(\Omega))}\leq
	\eps\|\Sn p^{(n)}\|_{L^{\infty}(0,\eps; L^1(\Omega))}\leq
	C \eps,
\end{align*}
and a similar estimate holds for $S p$ (e.g.~via Fatou's Lemma).
It follows that
\begin{align*}
	&\limsup_{n\to\infty}\|\Sn p^{(n)} - S p\|_{L^1(Q_\tf)}\\
	&\leq \limsup_{n\to\infty}\|\Sn p^{(n)} - S p\|_{L^1(Q_\tf^\eps)}
	+ \limsup_{n\to\infty}\|\Sn p^{(n)} - S p\|_{L^1(Q_\tf\backslash Q_\tf^\eps)}\\
	&\leq \sup_{n\in\N}\left( 
	\|\Sn p^{(n)}\|_{L^{1}(0,\eps; L^1(\Omega))} + 
	\|S p\|_{L^{1}(0,\eps; L^1(\Omega))}
	\right)\leq C\eps .
\end{align*}
This means that \eqref{p.cnv} holds.
\end{proof}
\subsection{Passing to the limit in \eqref{weak.1}--\eqref{weak.EBE} when $n \to \infty$.} 
In this subsection we will use the previous convergence results for passing to the limit in the variational entropy formulation.
Note that for a smooth solution the Radon measure $\xi$ equals zero.

\medskip

{\em Limit in equation \eqref{weak.1}.}
From 
\eqref{Est.J}, \eqref{Est.Srhot}, \eqref{cnv.aux} it follows
\begin{align*}
\Phi \Sn \rhon_i & \to \Phi S \rho_i \quad \mbox{strongly in }L^1(Q_\tf),\\
\veb J_i^{(n)} & \rightharpoonup \veb J_i\quad\mbox{weakly in }L^2(Q_\tf),
\end{align*}
for $i=1,\ldots,N$. 
Thanks to \eqref{Est.rhov} it holds
\begin{align*}
\vn\rightharpoonup \veb v\quad\mbox{weakly in } L^{\frac{2\beta}{\beta + 1}}(Q_\tf).%
\end{align*}
Since \eqref{rhoi.cnv} holds and (thanks to \eqref{hp.beta})
$\frac{1}{\gamma} + \frac{\beta + 1}{2\beta} < 1$, it follows
\begin{align*}
\rhon_i\veb v^{(n)}\rightharpoonup \rho_i\veb v\quad\mbox{weakly in 
$L^s(Q_\tf)$ for some $s>1$.}
\end{align*}
We must now identify $\veb v$.
We start by applying Lemma~\ref{Lemma_Gamma}.
Therefore \eqref{vow}, \eqref{mobility} lead to
\begin{align*}
-K\nabla\pn\rightharpoonup \frac{\veb v}{\lambda(S,T)}
\quad\mbox{weakly in }L^{{\frac{2\beta}{\beta + 1}}-\eps_1}(\Omega\times [\eps_2,\tf]),\quad\forall\eps_1,\eps_2>0.
\end{align*}
From \eqref{p.cnv.Qeps} it follows
\begin{align*}
-K\lambda(S,T)\nabla p = \veb v\quad\mbox{a.e.~in }\Omega\times [\eps_2,\tf],\quad\forall\eps_2>0.
\end{align*}
Being $\eps_2>0$ arbitrary and $\veb v\in L^{{\frac{2\beta}{\beta + 1}}}(Q_\tf)$, we deduce that \eqref{vow} holds a.e.~in $Q_\tf$. 
As a consequence of this fact and  \eqref{AJPP.1}, we also deduce
\begin{equation}
\label{wcnv.nap.1}
\frac{\lambda(\Sn,\Tn)^{1/2}}{(\Tn)^{1/2}}\nabla p^{(n)}
\rightharpoonup
\frac{\lambda(S,T)^{1/2}}{T^{1/2}}\nabla p\quad\mbox{weakly in }L^2(Q_\tf).
\end{equation} 

We will now identify the limits of the diffusion fluxes $\veb J_1,\ldots,\veb J_N$.
From \eqref{Est.B.2} it follows
\begin{align*}
	\Pi(\vec\mu^{(n)}/\Tn) \rightharpoonup\vec\zeta
	\quad\mbox{weakly in }L^2(0,\tf; H^1(\Omega)).
\end{align*}
The range of $\Pi:\R^N\to\R^N$ is weakly closed, hence $T\zeta = \Pi\vec\mu$ for some Lebesgue-measurable function $\mu : Q_T\to\R^N$.
Therefore
\begin{align}\label{muT.cnv}
	\Pi(\vec\mu^{(n)}/\Tn) \rightharpoonup
	\Pi(\vec\mu/T)
	\quad\mbox{weakly in }L^2(0,\tf; H^1(\Omega)).
\end{align}

Thanks to \eqref{hp.Ltilde}, \eqref{Ass.L}, it follows, for $i=1,\ldots,N$:
\begin{align}\label{lim.muT}
\sum_{j=1}^N \tilde L_{ij}(\vec\rhon,\Tn)\nabla(\mu_j^{(n)}/\Tn) \rightharpoonup
\sum_{j=1}^N \tilde L_{ij}(\vec\rho,T){\nabla(\mu_j/T)}
\quad\mbox{weakly in }L^2(Q_\tf).
\end{align}
We wish to show that the constitutive relation \eqref{chp.1} holds.
We point out that 
\begin{equation}
\label{rho.0}
\mbox{ for a.e. }(x,t)\in Q_\tf,\;
\rho(x,t)>0\quad\Rightarrow\quad\min_{i=1,\ldots,N}\rho_i(x,t)>0.
\end{equation}
This is a consequence of \eqref{est.relchp}. In fact, since $\rhon_i$ 
is a.e.~convergent in $Q_\tf$, Fatou's Lemma implies
\begin{align*}
\int_{Q_\tf\cap\{\rho>0\}}\log(\rho/\rho_i)dx\leq
\liminf_{n\to\infty}
\int_{Q_\tf\cap\{\rho>0\}}\log(\rhon/\rhon_i)dx\leq C,
\end{align*}
for $i=1,\ldots,N$. This means that $\rho_i>0$ on $Q_\tf\cap\{\rho>0\}$ for 
$i=1,\ldots,N$.

Thanks to the a.e.~convergence of $\rhon_i$, {$T^{(n)}$} and \eqref{rho.0},
 we deduce
\begin{align}\label{lim.muT.2}
\Pi\frac{\vec\mu^{(n)}}{\Tn}\to 
\Pi\frac{\vec\mu}{T}\quad\mbox{a.e.~on }Q_\tf\cap\{\rho>0\} ,\\
\nonumber
\left(\Pi\frac{\vec\mu}{T}\right)_i = 
{\left(\Pi\log\left( \frac{\vec{\rho}}{\rho} \right)\right)_i} \; \mbox{ a.e.~in }\; Q_\tf\cap\{\rho>0\} ,\quad i=1,\ldots,N.
\end{align}
\begin{remark}
\label{VacuumRemark}
We cannot exclude the possibility that $\rho=0$ on a set of positive measure. However, the constitutive relation \eqref{chp.1} is not meaningful in the vacuum region 
$\{\rho=0\}$, and consequently, we are unable to provide an explicit expression for the chemical potential vector 
$\vec{\mu}$ in this set. This limitation is a well-known issue in models describing fluid mixtures; see e.g.~\cite{DDGG20}.
\end{remark}
From \eqref{hp.Ltilde}, \eqref{Est.A}, \eqref{rhoi.cnv}, \eqref{T.cnv} we deduce that,
for $i=1,\ldots,N$,
\begin{align}
\label{Li0T.cnv}
\frac{\tilde L_{i0}(\rhonv,\Tn)}{\Tn}\to
\frac{\tilde L_{i0}(\vec\rho,T)}{T}\quad\mbox{strongly in }
L^p(Q_\tf),\quad \forall p<\infty,\\
\nonumber
\nabla\log\Tn\rightharpoonup\nabla\log T\quad\mbox{weakly in }L^2(Q_\tf).
\end{align}
Since $\frac{\tilde L_{i0}(\rhonv,\Tn)}{\Tn}\nabla\log\Tn$ is bounded in $L^2(Q_\tf)$, we deduce (up to subsequences)
\begin{align*}
\frac{\tilde L_{i0}(\rhonv,\Tn)}{\Tn}\nabla\log\Tn\rightharpoonup
\frac{\tilde L_{i0}(\vec\rho,T)}{T}\nabla\log T\quad\mbox{weakly in }
L^2(Q_\tf),
\end{align*}
for $i=1,\ldots,N$, which means
\begin{align}\label{lim.cross.1}
\tilde L_{i0}(\rhonv,\Tn)\nabla\frac{1}{\Tn}\rightharpoonup
\tilde L_{i0}(\vec\rho,T)\nabla\frac{1}{T}\quad\mbox{weakly in }
L^2(Q_\tf).
\end{align}
From \eqref{lim.muT}, \eqref{lim.muT.2}, \eqref{lim.cross.1} we obtain
\begin{align}
\label{lim.J}
\veb{J}_i &= L_{i0} \nabla \Big(\frac{1}{T}\Big) 
- \sum_{k=1}^{N} L_{ik} \nabla\zeta_k, \quad \mbox{a.e.~on }Q_\tf ,\\
L_{ij} &= \tilde L_{ij}(\vec\rho,T),\quad
L_{i0} = \tilde L_{i0}(\vec\rho,T)\quad
\mbox{a.e.~on }Q_\tf .\nonumber
\end{align}
for $i,j=1,\ldots,N$. 

Let us now focus on the reaction terms.
From \eqref{Est.r} it follows that
\begin{align*}
r_i^{(n)}\rightharpoonup r_i\quad\mbox{weakly in }L^{2}(Q_\tf),
\quad i=1,\ldots,N.
\end{align*}
However, since $r_i^{(n)} = \tilde r_i(\Pi(\vec\mu^{(n)}, \Tn)$, the continuity of $\tilde r_i$ as well as
\eqref{rho.cnv}, \eqref{T.cnv.2}, \eqref{lim.muT.2} we deduce
\begin{equation}
\label{lim.r}
r_i = r_i(\Pi(\vec\mu/T))\;\mbox{ a.e.~in } \; Q_\tf\cap\{\rho>0\}.
\end{equation}
Finally, we point out that the continuous Sobolev embedding $H^{1}(\Omega)\hookrightarrow H^{1/2}(\pa\Omega)$ yields the convergence of the boundary integral:
\begin{align*}
\sum_{k=1}^N b_{ik}\left(\frac{\mu^{(n)}_k}{\Tn} - \frac{\mu_{0,k}}{T_0}  \right)\rightharpoonup
\sum_{k=1}^N b_{ik}\left(\zeta_k - \frac{\mu_{0,k}}{T_0}\right)
\quad\mbox{ weakly in }L^2(\pa\Omega\times (0,\tf)).
\end{align*}
This shows that \eqref{weak.1} holds {in the limit as $n\to\infty$}.

\medskip

{\em Limit in equation \eqref{weak.EBE}.}
Due to Lemma \ref{str.convergences}, equations \eqref{S.cnv}-\eqref{T.cnv} we get that
\[ \Phi \Sn  (\rho\eta)^{(n)}+ (1-\Phi)(\rho\eta)^{(n)}_s \to \Phi S (\rho \eta) + (1-\Phi)(\rho \eta)_s 
\quad\mbox{ strongly in } L^1(Q_\tf). \]
Let us now turn our attention to $\veb q^{(n)}$. Recall from   \eqref{flux.q} that
\begin{align*}
	\frac{\veb q^{(n)}}{\Tn} = 
	-\frac{\kappa(\Tn)}{\Tn}\nabla\Tn 
	+ \sum_{j=1}^N\frac{\tilde L_{0j}(\rhonv,\Tn)}{\Tn}\nabla\frac{\mu_j^{(n)}}{\Tn} .
\end{align*}
From \eqref{Est.kappaT} it follows
\begin{align}\label{q.cnv.1}
\frac{\veb q^{(n)}}{\Tn}\rightharpoonup\tilde{\veb q}\quad
\mbox{weakly in }L^{\frac{2+3\beta}{1+3\beta}}(Q_\tf).
\end{align}
From \eqref{hp.Ltilde}, \eqref{muT.cnv}, \eqref{Li0T.cnv} and 
the boundedness of $\sum_{j=1}^N\frac{\tilde L_{0j}(\rhonv,\Tn)}{\Tn}\nabla\frac{\mu_j^{(n)}}{\Tn}$
in $L^2(Q_\tf)$ we deduce
\begin{align}
\sum_{j=1}^N
\frac{\tilde L_{0j}(\rhonv,\Tn)}{\Tn}\nabla\frac{\mu_j^{(n)}}{\Tn}
\rightharpoonup
\sum_{j=1}^N\frac{\tilde L_{0j}(\vec\rho,T)}{T}\nabla
	\frac{\mu_j}{T}
\quad\mbox{weakly in }L^2(Q_\tf).
\label{L0j_cvg}
\end{align}
On the other hand, \eqref{Est.kappaT}, \eqref{T.cnv} imply
\begin{align}
\frac{\kappa(\Tn)}{\Tn}\nabla\Tn \rightharpoonup
\frac{\kappa(T)}{T}\nabla T \quad\mbox{weakly in }L^{\frac{2+3\beta}{1+3\beta}}(Q_\tf).
\label{kappaT_cvg}
\end{align}
We conclude 
\begin{equation}
\label{qtilde}
\tilde{\veb q} = 
-\frac{\kappa(T)}{T}\nabla T 
+ \sum_{j=1}^N\frac{\tilde L_{0j}(\vec\rho,T)}{T}\nabla\frac{\mu_j}{T} 
=\frac{\veb q}{T},\quad
\mbox{a.e.~in }Q_\tf.
\end{equation}
From \eqref{Est.muJ}, \eqref{lim.muT.2} and \eqref{lim.J}  it follows
\begin{align*}
-\sum_{i=1}^N\frac{\mu_i^{(n)}}{\Tn}\veb J_i^{(n)}\rightharpoonup
\Xi\quad\mbox{weakly in }L^{\frac{2a}{a+2}}(Q_\tf),\quad
\Xi = -\sum_{i=1}^N\frac{\mu_i}{T}\veb J_i\quad\mbox{a.e.~in }Q_\tf\cap\{\rho>0\}.
\end{align*}
Finally, consider the quantity
\begin{align*}
\xi^{(n)}
&=\frac{K}{\Tn}\lambda(\Sn,\Tn)|\nabla p^{(n)}|^2 +
\sum_{i,j=1}^N L_{ij}^{(n)} \nabla \big( \frac{\mu_i^{(n)}}{\Tn} \big) \cdot \nabla \big( \frac{\mu_j^{(n)}}{\Tn} \big) + 
{\kappa(T^{(n)}) \big| \nabla \log\Tn  \big|^2 }  \nonumber\\ 
&\qquad - \Phi \frac{1}{\Tn} f'(\Sn) \big( \partial_t \Sn \big)^2
-\sum_{i=1}^N r_i^{(n)}\frac{\mu_i^{(n)}}{\Tn}\\
&-\frac{K}{T}\lambda(S,T)|\nabla p|^2 
-\sum_{i,j=1}^N L_{ij} \nabla \big( \frac{\mu_i}{T} \big) 
\cdot \nabla \big( \frac{\mu_j}{T} \big)  
- {\kappa(T) \big| \nabla \log T \big|^2 } \nonumber\\
&\qquad +\Phi \frac{1}{T} f'(S) \big( \partial_t S \big)^2
+\sum_{i=1}^N r_i\frac{\mu_i}{T}.
\end{align*}
By Lemma \ref{lem.ap.est} it is clear that $\xi^{(n)}$ is bounded in $L^1(Q_\tf)$ and therefore
also in $\mathcal M(\overline{Q_\tf})$, leading to
$\xi^{(n)}\rightharpoonup^*\xi$ weakly* in $\mathcal M(\overline{Q_\tf})$. 

We will prove that $\xi$ is a nonnegative measure.
Let $\varphi\in C^0(\overline{Q_\tf})$, $\varphi\geq 0$
in $Q_\tf$.	
From \eqref{wcnv.nap.1} and the weak lower semicontinuity of 
the $L^2(Q_\tf)$-norm we deduce
\begin{align}\label{Fatou.1}
\liminf_{n\to\infty}\int_{Q_\tf}\frac{K}{\Tn}\lambda(\Sn,\Tn)|\nabla p^{(n)}|^2 \varphi dx dt \geq
\int_{Q_\tf}\frac{K}{T}\lambda(S,T)|\nabla p|^2 
\varphi dx dt.
\end{align}
In the same way, from \eqref{L0j_cvg} and \eqref{kappaT_cvg} we get
\begin{align}
\label{Fatou.3}
\liminf_{n\to\infty}\int_{Q_\tf}
\sum_{i,j=1}^N L_{ij}^{(n)} \nabla \big( \frac{\mu_i^{(n)}}{\Tn} \big) \cdot \nabla \big( \frac{\mu_j^{(n)}}{\Tn} \big) \varphi dx dt &\geq 
\int_{Q_\tf}
\sum_{i,j=1}^N L_{ij} \nabla \big( \frac{\mu_i}{T} \big) \cdot \nabla \big( \frac{\mu_j}{T} \big) \varphi dx dt,\\
\label{Fatou.4}
\liminf_{n\to\infty}\int_{Q_\tf}
 { \kappa(T^{(n)}) \big| \nabla \log\Tn \big|^2}
\varphi dx dt &\geq \int_{Q_\tf}
{ \kappa(T) \big| \nabla \log T \big|^2 }
\varphi dx dt.
\end{align}

Furthermore, from \eqref{Est.S} it follows that
\begin{align*}
\frac{1}{(\Tn)^{1/2}} \sqrt{-f'(\Sn)} \partial_t \Sn
\rightharpoonup
\omega\quad\mbox{weakly in }L^2(Q_\tf).
\end{align*}
Since \eqref{T.cnv.2} holds we deduce
\begin{align*}
\sqrt{-f'(\Sn)} \partial_t \Sn
\rightharpoonup
\sqrt{T}\omega\quad\mbox{weakly in }L^1(Q_\tf).
\end{align*}
We will now identify the weak limit $\omega$.
Let $\phi\in C^\infty_c(Q_\tf)$ arbitrary, and let us consider
\begin{align*}
\int_{Q_\tf} \sqrt{-f'(\Sn)} \partial_t \Sn \phi \, dx dt & =\int_{Q_\tf}\partial_t F(\Sn) \phi \, dx dt
= -\int_{Q_\tf} F(\Sn)  \partial_t \phi \, dx dt  
\end{align*}
where $F(s) = \int_0^s   \sqrt{-f'(\sigma)}\, d\sigma$ for $0\leq s< 1$. Cauchy-Schwartz inequality and the fundamental theorem of calculus yield
$$
F(s)\leq \sqrt{s}\sqrt{f(0)-f(s)}\quad 0\leq s<1.
$$
From \eqref{bounds.fSn}, the fact that $0\leq S^{(n)}\leq 1$ and the above relation it follows that $F(S^{(n)})$ is bounded in $L^{2q}(Q_\tf)$. Since $S^{(n)}\to S$ a.e.~in $Q_\tf$ and $S<1$ a.e.~in $Q_\tf$ (given \eqref{bounds.fSn}), we deduce by Proposition \ref{delaValle} that $F(S^{(n)})\to F(S)$ strongly in $L^1(Q_\tf)$. We deduce
$$
\int_{Q_\tf} \sqrt{-f'(\Sn)} \partial_t \Sn \phi \, dx dt\to 
-\int_{Q_\tf} F(S) \partial_t \phi \, dx dt = 
\int_{Q_\tf} \sqrt{-f'(S)} \partial_t S \phi \, dx dt
$$
hence $\sqrt{T}\omega = \sqrt{-f'(S)} \partial_t S$. It follows
\begin{align*}
\frac{1}{(\Tn)^{1/2}} \sqrt{-f'(\Sn)} \partial_t \Sn
\rightharpoonup
\frac{1}{T^{1/2}} \sqrt{-f'(S)} \partial_t S
\quad\mbox{weakly in }L^2(Q_\tf).
\end{align*}
Once again, the above relation and the weak lower semicontinuity of the $L^2(Q_\tf)$ norm lead to
\begin{align}\label{Fatou.2}
\liminf_{n\to\infty}\int_{Q_\tf}
- \Phi \frac{1}{\Tn} f'(\Sn) \big( \partial_t \Sn \big)^2 
\varphi dx dt \geq
\int_{Q_\tf} 
-\Phi \frac{1}{T} f'(S) \big( \partial_t S \big)^2 
\varphi dx dt .
\end{align}
Finally, assumption \eqref{Add_Fatou.5}
implies that
\begin{align}
\label{Fatou.5}
\liminf_{n\to\infty}\int_{Q_\tf}
-\sum_{i=1}^N r_i^{(n)}\frac{\mu_i^{(n)}}{\Tn} \varphi dx dt
&\geq\int_{Q_\tf}
-\sum_{i=1}^N r_i\frac{\mu_i}{T} \varphi dx dt.
\end{align}
From \eqref{Fatou.1}--\eqref{Fatou.5} and the definition of $\xi$ we conclude
\begin{align*}
\langle\xi,\varphi\rangle = 
 \lim_{n\to\infty}\langle\xi^{(n)},\varphi\rangle\geq 0.
\end{align*}
This means that $\xi$ is nonnegative. Therefore \eqref{weak.EBE} holds.

\medskip

{\em Limit in equation \eqref{weak.2}.}
We will now prove that \eqref{weak.2} holds in the limit $n\to\infty$. 
By using \eqref{Est.A} and  \eqref{Est.T} one can prove, as in  Lemma \ref{lema-TT} that
$ \|T\|_{L^2(0,\tf; H^1(\Omega))} \leq C$.
From  \eqref{T.cnv.2} it follows that
$\Tn\rightharpoonup T$ weakly in $L^2(0,\tf; H^1(\Omega))$.
Thanks to the Sobolev embedding $H^1(\Omega)\hookrightarrow H^{1/2}(\pa\Omega)$ this implies that
$\Tn\rightharpoonup T$ weakly in $L^2(0,\tf; H^{1/2}(\pa\Omega))$. In particular
\begin{align*}
\int_0^t\int_{\pa\Omega}(\Tn - T_0)d\sigma dt'\to
\int_0^t\int_{\pa\Omega}(T - T_0)d\sigma dt',\quad n\to\infty .
\end{align*}
Furthermore $E_s(\Tn) = c_s\Tn\to c_s T = E_s(T)$ strongly in $L^2(Q_\tf)$ thanks to 
\eqref{Skel.Ener}, \eqref{T.cnv.2}, while the term (recall \eqref{def.Eint})
\begin{align*}
E_{int}(\Sn) = \int_{\Sn}^1 P_c(s) ds	
\end{align*}
can be estimated from \eqref{hp.Pc.bound} and \eqref{PC.Sw.Est1}  
through a similar argument as in \eqref{Est.SPc}, implying that
\begin{align*}
E_{int}(\Sn)\to E_{int}(S)\quad\mbox{strongly in }L^1(Q_\tf).
\end{align*}
Therefore we must now only show that the term (recall \eqref{eng.1})
\begin{align*}
\Sn (\rho e)^{(n)} = \Sn  \left( (\rhon)^\gamma + c_w\rhon \Tn + p_{at} \right)
\end{align*}
is strongly convergent in $L^1(Q_\tf)$. However, from \eqref{thp.1} we obtain
\begin{align*}
(\gamma-1)\Sn (\rho e)^{(n)} - \Sn p^{(n)} = 
[c_w(\gamma-1)-1]\Sn\rhon\Tn + \gamma {\Sn}p_{at} ,
\end{align*}
which, thanks to \eqref{S.cnv}, \eqref{rho.cnv}, \eqref{T.cnv.2}, \eqref{p.cnv}, leads to
\begin{align*}
(\gamma-1)\Sn (\rho e)^{(n)} \to S p + 
[c_w(\gamma-1)-1]S\rho T + \gamma {S}p_{at} = 
(\gamma-1)S (\rho e)\\
\mbox{strongly in }L^1(Q_\tf).
\end{align*}
We conclude that \eqref{weak.2} holds {in the limit  $n\to \infty$}.
\medskip

 {\em Limit in equation \eqref{PC.Sw}. We will now show that the saturation balance equation \eqref{PC.Sw} holds in the limit $n\to\infty$. }

We have by assumption
\begin{align*}
	\pa_t f(\Sn) + P_c(\Sn) + p^{(n)} = 0,\quad t>0,\quad\mbox{a.e.~in }\Omega .
\end{align*}
Multiplying the equation times $\Sn$ leads to
\begin{align*}
	\pa_t F(\Sn) + \Sn P_c(\Sn) + \Sn p^{(n)} = 0,\quad t>0,\quad\mbox{a.e.~in }\Omega ,
\end{align*}
where
 $F(S) = -\int_S^{1/2} s_1 f'(s_1)d s_1$.
Integrating the above equation against a test function $\varphi\in C^1_c(Q_\tf)$ yields
\begin{align}\label{eq.F}
	-\int_0^\tf\int_\Omega F(\Sn)\pa_t\varphi dx dt 
	+ \int_0^\tf\int_\Omega \left(\Sn P_c(\Sn) + \Sn p^{(n)} \right)\varphi dx dt = 0.
\end{align}
Since
 {$F(S) =  -f(1/2)/2 + Sf(S) + \int_S^{1/2} f(s_1)d s_1$} 
 and $f\in C^0([0,1))$, then also $F\in C^0([0,1))$. Furthermore, note that
 $| F(S) | \leq | S - 1/2 | \cdot | f(S) - f(1/2) | \leq C(1+|f(S)|)$.
 Therefore from \eqref{S.cnv} we deduce
$F(\Sn)\to F(S)$ strongly in $L^{q-\eps}(Q_\tf)$ for every $\eps>0$, with $q$ taken from \eqref{Est.S}.
On the other hand, \eqref{hp.Pc.bound} and \eqref{PC.Sw.Est1} imply
\begin{align}\label{Est.SPc}
	\begin{cases}
		\|\Sn P_c(\Sn)\|_{L^{\infty}(Q_\tf)} \leq
		C, & k_p\leq 1,\\
		\|\Sn P_c(\Sn)\|_{L^{\frac{k_p}{k_p-1}}(Q_\tf)} \leq
		C\| P_c(\Sn)^{1-1/k_p} \|_{L^{\frac{k_p}{k_p-1}}(Q_\tf)} 
		\leq C, & k_p>1
	\end{cases} .
\end{align}
Since {$k_p/(k_p-1)>1$} and \eqref{S.cnv} holds, we deduce that 
$\Sn P_c(\Sn)\to S P_c(S)$ strongly in $L^1(Q_\tf)$.
Using \eqref{S.cnv} and \eqref{p.cnv} we can take the limit $n\to\infty$ in \eqref{eq.F} and get
\begin{align*}
	-\int_0^\tf\int_\Omega F(S)\pa_t\varphi dx dt 
	+ \int_0^\tf\int_\Omega \left(S P_c(S) + S p \right)\varphi dx dt = 0,\quad\forall\varphi\in C^1_c(Q_\tf),
\end{align*}
which is the weak formulation of
\begin{align*}
	\pa_t F(S) + S P_c(S) + S p = 0,\quad t>0,\quad\mbox{a.e.~in }\Omega .
\end{align*}
Since $F'(S) = S f'(S)$ and $S>0$ a.e.~in $Q_\tf$, dividing the above equation times $S$ yields \eqref{PC.Sw}.
This finishes the proof of Theorem \ref{thm.ws}.


\section{Appendix}
For the convenience of the reader, 
we present here some proofs and results which are rather technical, but are nevertheless needed for completeness. 

\paragraph*{{\bf Proof of Lemma~\ref{lema-TT}}}

Let $1 \leq r < \beta/2$ be arbitrary.  In this case we observe that
\begin{align*}
	\nabla T^r =  \chf{(0,1)}(T) r T^r \nabla\log T + \chf{[1,\infty)}(T)\frac{2r}{\beta}T^{r-\beta/2}\nabla T^{\beta/2}
\end{align*}
is bounded in $L^2(Q_\tf)$ given the uniform bounds for $\nabla\log T$ and $\nabla T^{\beta/2}$ in $L^2(Q_\tf)$. If $T^r \in L^1(\Qt)$ then using Poincar\'e's Lemma we can estimate $T^r$ in $L^2(0,\tf; H^1(\Omega))$ and therefore (via Sobolev embedding) in $L^2(0,\tf; L^6(\Omega))$. 
More precisely, we conclude that $\forall r \in [1,\beta/2)$ one has
\begin{equation}
	\| T\|_{L^{2r}(0,\tf;L^{6r}(\Omega))}^r \leq C\Big( 
	\| \nabla \log T\|_{L^{2}(\Qt)} + 	\| \nabla T^{\beta/2} \|_{L^{2}(\Qt)} + 
		\| T\|_{L^{2r}(0,\tf;L^{r}(\Omega))}^r \Big).
		\label{Ap.Pom.Int.1}
\end{equation}
Standard interpolation argument for $L^p$ spaces gives
\begin{equation*}
	\|T\|_{L^{2r_1}(0,\tf; L^{r_1}(\Omega))}^{r_1} \leq 
    \| T\|_{L^{\infty}(0,\tf; L^1(\Omega))}^{r_1-r} \|T\|_{L^{2r}(0,\tf; L^{6r}(\Omega))}^r, \quad 
	r_1 = r + \frac{5}{6}.
\end{equation*}
From \eqref{Ap.Pom.Int.1} we get the following recursion:
\begin{equation} 
\begin{split}
	&\|T\|_{L^{2r_1}(0,\tf; L^{r_1}(\Omega))}^{r_1} \\
  &\qquad  \leq 
 C \| T\|_{L^{\infty}(0,\tf; L^1(\Omega))}^{5/6} 
 \Big( 
	\| \nabla \log T\|_{L^{2}(\Qt)} + 	\| \nabla T^{\beta/2} \|_{L^{2}(\Qt)} + 
		\| T\|_{L^{2r}(0,\tf;L^{r}(\Omega))}^r \Big).
  \end{split}
	\label{Ap.Pom.Int.2}
\end{equation}
Iterating \eqref{Ap.Pom.Int.2}  from $r=1$ until $r < \beta/2$ finally yields the bound on $T^{\beta/2}$ in $L^2(0,\tf; L^1(\Omega))$. Then, the statement follows by an application of Poincar\'e's Lemma .
\paragraph*{{\bf Proof of Proposition \ref{prop.5}}}

Estimate \eqref{Est.r} follows directly from \eqref{React_1} and $\eqref{Add_Psi}_2$ and apriori estimate \eqref{Est.B.2}.
Next, since \eqref{vow}, \eqref{mobility} hold and $\nu$ is uniformly positive, we deduce
\begin{align*}
	|\veb v| &\leq C \sqrt{\lambda(S,T)T}\, \sqrt{\frac{\lambda(S,T)}{T}}|\nabla p| 
	\leq C  T^{1/2}\, \sqrt{\frac{\lambda(S,T)}{T}}|\nabla p|.
\end{align*}
From \eqref{hp.kr} and H\"older's inequality
(notice that $ \frac{1 + \beta}{2\beta}= \frac{1}{2\beta} + \frac{1}{2}$) we get
\begin{align*}
	\|\veb v \|_{ L^{\frac{2\beta}{1+ \beta}}(Q_\tf)}
	\leq 	C 	\|\sqrt T\|_{L^{2\beta}(Q_\tf)}
	\left\|	\sqrt{\frac{\lambda(S,T)}{T}}\nabla p\right\|_{L^2(Q_\tf)} = 
	C \| T^{\frac{\beta}{2}}\|_{L^{2}(Q_\tf)}^{\frac{1}{\beta}} \left\|	\sqrt{\frac{\lambda(S,T)}{T}}\nabla p\right\|_{L^2(Q_\tf)} .
\end{align*}
From this estimates and \eqref{Est.T}, \eqref{AJPP.1} we obtain \eqref{Est.rhov}. Furthermore, 
we have
\begin{align}\label{Est.aux.rhov}
 \| \rho \veb v\|_{L^m(Q_\tf)}	\leq \| \rho\|_{L^{\gamma}(Q_\tf)} \|\veb v \|_{ L^{\frac{2\beta}{1+ \beta}}(Q_\tf)},
 \quad m = \frac{2\beta\gamma}{\beta(2+\gamma)+\gamma},
\end{align} 
where $m>1$ due to \eqref{hp.beta}.

Let us now estimate $\veb J_i$. From \eqref{flux.J}, 
\eqref{hp.Ltilde} and \eqref{Ass.L} it follows
\begin{align*}
|\veb J_i|\leq C(|\nabla\log T| + |\nabla\Pi(\vec {\mu}/T)|),\quad i=1,\ldots,N,
\end{align*}
so from \eqref{Est.A}, \eqref{Est.B.2} we deduce \eqref{Est.J}.
From \eqref{cons_w}, \eqref{Est.aux.rhov}, \eqref{Est.J} we easily conclude
that \eqref{Est.flux.1} and \eqref{Est.Srhot} hold.
From \eqref{Est.E.LinfL1.1} $S p$ is bounded in $L^\infty(0,\tf; L^1(\Omega))$. As a consequence, multiplying \eqref{PC.Sw} by $S$ and taking the $L^\infty(0,\tf; L^1(\Omega))$
norm leads to \eqref{Est.FSt}. 

\paragraph*{{\bf Proof of Proposition \ref{prop.6}}}

From \eqref{entropy} it follows
\begin{align*}
\|S(\rho\eta)\|_{L^{\frac{2\gamma}{\gamma+2}}(Q_\tf)} \leq 
\sum_{i=1}^N\|S\rho_i\log\rho_i\|_{L^{\frac{2\gamma}{\gamma+2}}(Q_\tf)}
+ c_w\|S\rho (\log T + 1)\|_{L^{\frac{2\gamma}{\gamma+2}}(Q_\tf)}.
\end{align*} 
However, for $\delta\in (0,\gamma/2]$, since $S\leq S^{(1+\delta)/\gamma}$ and
$|\rho_i\log\rho_i|\leq C(1+\rho^{1+\delta})$, it holds
\begin{align*}
\sum_{i=1}^N\|S\rho_i\log\rho_i\|_{L^{\frac{2\gamma}{\gamma+2}}(Q_\tf)}
&\leq C + C\|S^\frac{1+\delta}{\gamma}\rho^{1+\delta}\|_{L^{\frac{2\gamma}{\gamma+2}}(Q_\tf)}\\
& = C + C \|S^{1/\gamma}\rho\|_{L^{\frac{2\gamma(1+\delta)}{\gamma+2}}(Q_\tf)}^{
	1+\delta}.
\end{align*}
Since $\gamma>2$ and $0<\delta\leq \gamma/2$, from \eqref{Est.E.LinfL1.1} it follows
\begin{align*}
	\sum_{i=1}^N\|S\rho_i\log\rho_i\|_{L^{\frac{2\gamma}{\gamma+2}}(Q_\tf)}
	&\leq C .
\end{align*}
Moreover, H\uml older's inequality yields
\begin{align*}
\|S\rho (\log T + 1)\|_{L^{\frac{2\gamma}{\gamma+2}}(Q_\tf)}
\leq \|S\rho\|_{L^\gamma(Q_\tf)}\|\log T + 1\|_{L^2(Q_\tf)}\leq C
\end{align*}
thanks to \eqref{Est.E.LinfL1.1}, \eqref{Est.A}. We conclude
\begin{equation}
\label{Est.rhoeta}
\|S(\rho\eta)\|_{L^{\frac{2\gamma}{\gamma+2}}(Q_\tf)} \leq C.
\end{equation}
From \eqref{Skel.Entr}, \eqref{Est.A} it follows that $(\rho\eta)_s$ is bounded in $L^2(Q_\tf)$, so we get \eqref{Est.entr}.

We will now find an estimate for the entropy flux.
We begin by considering 
\begin{align*}
(\rho\eta)\veb v = 
K\left(\sum_{i=1}^N\rho_i\log\rho_i - c_w\rho (\log T + 1)\right)
\lambda(S,T) \nabla p ,
\end{align*}
where the above equality holds thanks to \eqref{entropy}, \eqref{vow}.
Let $s\in\R$ be such that
\begin{equation}
\label{def.s}
s>1,\qquad
\frac{1}{\gamma} + \frac{1}{2\beta} + \frac{1}{2} < \frac{1}{s} .
\end{equation}
The above definition makes sense since \eqref{hp.beta} holds.
It follows via H\uml older's inequality
\begin{align*}
\|(\rho\eta)\veb v\|_{L^s(Q_\tf)} \leq 
C\sum_{i=1}^N\|\rho_i (\log\rho_i) \sqrt{\lambda(S,T)T}\|_{L^{\frac{2s}{2-s}}(Q_\tf)}
\left\|\sqrt\frac{\lambda(S,T)}{T}\nabla p\right\|_{L^2(Q_\tf)}\\
+C\|\rho (\log T + 1)\sqrt{\lambda(S,T)T}\|_{L^{\frac{2s}{2-s}}(Q_\tf)}
\left\|\sqrt\frac{\lambda(S,T)}{T}\nabla p\right\|_{L^2(Q_\tf)} .
\end{align*}
However, since $\mu$ is uniformly positive and \eqref{AJPP.1} holds,
we get
\begin{align*}
\|(\rho\eta)\veb v\|_{L^s(Q_\tf)} \leq 
	C\sum_{i=1}^N\|\rho_i (\log\rho_i) \sqrt{k_r(S)T}\|_{L^{\frac{2s}{2-s}}(Q_\tf)}\\
	+C\|\rho (\log T + 1)\sqrt{k_r(S)T}\|_{L^{\frac{2s}{2-s}}(Q_\tf)} .
\end{align*}
Now, since for every $\delta>0$ there exists $C_\delta>0$ such that
$x|\log x| \leq C_\delta(1+x^{1+\delta})$ for $x>0$, H\uml older's inequality and \eqref{def.s} allow us to state
\begin{align*}
\|(\rho\eta)\veb v\|_{L^s(Q_\tf)} \leq 
C_\delta\sum_{i=1}^N\|(1+\rho_i^{1+\delta}) \sqrt{k_r(S)} \, T^{1/2}\|_{L^{\frac{2s}{2-s}}(Q_\tf)}\\
+C_\delta\|\rho\sqrt{k_r(S)} (1+T^{(1+\delta)/2}) \|_{L^{\frac{2s}{2-s}}(Q_\tf)}\\
\leq C_\delta\sum_{i=1}^N\| (1+\rho_i^{1+\delta}) \sqrt{k_r(S)} \|_{L^{\gamma/(1+\delta)}(Q_\tf)}
\|T^{1/2}\|_{L^{2\beta}(Q_\tf)} \\
+C_\delta\| \rho\sqrt{k_r(S)}\|_{L^\gamma(Q_\tf)}
\| (1+T^{(1+\delta)/2}) \|_{L^{2\beta/(1+\delta)}(Q_\tf)},
\end{align*}
for some $\delta>0$ small enough. Assumption {\bf (H2)} and bounds \eqref{Est.p}, \eqref{Est.T} allow us to conclude
\begin{equation}
\label{Est.rhoetav}
\exists s>1:\quad
\|(\rho\eta)\veb v\|_{L^s(Q_\tf)} \leq C.
\end{equation}
Let us then consider
\begin{align*}
-\sum_{i=1}^N\frac{\mu_i}{T}\veb J_i = -\sum_{i=1}^N\frac{\mu_i}{T}L_{i0}\nabla\frac{1}{T}
+ \sum_{i,j=1}^N\frac{\mu_i}{T}L_{ij}\nabla\frac{\mu_j}{T},
\end{align*}
where the above equality comes from \eqref{flux.J}. Since \eqref{hp.Ltilde} holds and
$L_{ij}$ is symmetric and positive semidefinite, we obtain via Cauchy-Schwartz and using \eqref{Ass.L} and \eqref{Li0.Sum} that
\begin{align*}
\left|-\sum_{i=1}^N\frac{\mu_i}{T}\veb J_i \right| &\leq
C\left|\Pi\frac{\vec\mu}{T}\right| |\nabla\log T|
+ C\left( \sum_{i,j=1}^N L_{ij}\frac{\mu_i\mu_j}{T^2} \right)^{\frac{1}{2}}
\left( \sum_{i,j=1}^N L_{ij}\frac{\nabla\mu_i\cdot\nabla\mu_j}{T^2} \right)^{\frac{1}{2}}\\
&\leq C\left|\Pi\frac{\vec\mu}{T}\right|\left(
|\nabla\log T| + \left|\nabla\Pi\frac{\vec\mu}{T}\right|
\right).
\end{align*}
Given that \eqref{Est.A}, \eqref{Est.B.2} hold, we obtain
\begin{equation}
	\label{Est.muJ}
	\left\|\sum_{i=1}^N\frac{\mu_i}{T}\veb J_i \right\|_{L^{\frac{2a}{a+2}}(Q_\tf)}\leq C.
\end{equation}
Finally, let us consider 
\begin{align*}
\frac{\veb q}{T} = -\kappa(T)\nabla\log T + 
\sum_{j=1}^N \frac{L_{0j}}{T}\nabla\frac{\mu_j}{T} ,
\end{align*}
where the above equality comes from \eqref{flux.q}, \eqref{heat_cond}.
It follows from \eqref{kappa},  that
\begin{align*}
\left|\kappa(T)\nabla\log T\right| \leq C|\nabla\log T| + C T^{\beta/2}|\nabla T^{\beta/2}|,
\end{align*}
while \eqref{hp.Ltilde} and \eqref{Li0.Sum} imply
\begin{align*}
\left|\sum_{j=1}^N \frac{L_{0j}}{T}\nabla\frac{\mu_j}{T}\right|\leq
C \left|\nabla\Pi\frac{\vec\mu}{T}\right|.
\end{align*}
From \eqref{Est.A}--\eqref{Est.B.2} and H\uml older's inequality we conclude
\begin{align*}
\left\|\kappa(T)\nabla\log T\right\|_{L^{\frac{2+3\beta}{1+3\beta}}(Q_\tf)}
+
\left\|\sum_{j=1}^N \frac{L_{0j}}{T}\nabla\frac{\mu_j}{T}\right\|_{L^2(Q_\tf)}\leq C,
\end{align*}
which proves \eqref{Est.kappaT} and it leads to
\begin{equation}
\label{Est.qT}
\left\|\frac{\veb q}{T}\right\|_{L^{\frac{2+3\beta}{1+3\beta}}(Q_\tf)}
\leq C .
\end{equation}
Putting \eqref{Est.rhoetav}, \eqref{Est.muJ}, \eqref{Est.qT} together
allows us to obtain the estimate for the entropy flux \eqref{Est.entrflux}.
\paragraph*{{\bf Proof of Proposition \ref{prop.7}}}
Let us first find an estimate for $\nabla\rho^\gamma$. 
Eq.~\eqref{thp.1} yields:
\begin{align*}
	\nabla p 
	= (T + \gamma(\gamma-1)\rho^{\gamma-1})\nabla\rho + \rho\nabla T ,
\end{align*}
which implies
\begin{align*}
	|\nabla\rho^\gamma|\leq |(T + \gamma(\gamma-1)\rho^{\gamma-1})\nabla\rho|\leq
	\rho |\nabla T| + |\nabla p| .
\end{align*}
From \eqref{hp.kr} and the uniform boundedness of $\nu$ it follows
\begin{align*}
	\sqrt{k_r(S)}|\nabla\rho^\gamma|\leq C \Big( \rho S^{1/\gamma}|\nabla T|
	+ T^{1/2}\sqrt\frac{\lambda(S,T)}{T} |\nabla p|\Big).
\end{align*}
From \eqref{Est.E.LinfL1.1}, \eqref{Est.A}, \eqref{Est.T}, \eqref{AJPP.1}, 
using H\uml older's inequality, it follows
\begin{align*}
\|\rho S^{1/\gamma}\nabla T\|_{L^{\frac{2\gamma}{2+\gamma}}(Q_\tf)} \leq 
\|\rho S^{1/\gamma}\|_{L^\gamma(Q_\tf)}\|\nabla T\|_{L^2(Q_\tf)}\leq C,\\
\left\|
 T^{1/2}\sqrt\frac{\lambda(S,T)}{T} \nabla p
\right\|_{L^{\frac{2\beta + 4/3}{\beta + 5/3}}(Q_\tf)}\leq
\left\|T^{1/2}\right\|_{L^{2\beta + 4/3}(Q_\tf)}\left\|
\sqrt\frac{\lambda(S,T)}{T} \nabla p
\right\|_{L^{2}(Q_\tf)}\leq C,
\end{align*}
so we deduce
\begin{align*}
\|\sqrt{k_r(S)}\nabla\rho^\gamma\|_{L^{a_2}(Q_\tf)}\leq C,\quad
a_2\equiv\min\left\{\frac{2\gamma}{2+\gamma},
\frac{2\beta + 4/3}{\beta + 5/3}  \right\} > 1.
\end{align*}
In particular we deduce
\begin{align}\label{est.narho.1}
\|\sqrt{k_r(S)}\nabla G(\rho^\gamma)]\|_{L^{a_2}(Q_\tf)}\leq C[G],
\end{align}
where $G\in W^{1,\infty}(\R_+)$ is arbitrary, and the constant $C[G]>0$ depends on $G$.

On the other hand, since $1/\gamma + 1/q < 1$,
from \eqref{hp.krf}, \eqref{Est.S} it follows
\begin{align*}
\|G(\rho^\gamma)\nabla \sqrt{k_r(S)}\|_{L^{\infty}(0,\tf; L^q(\Omega))} &\leq c_f' \| G(\rho^\gamma) \|_{L^\infty(Q_\tf)}\|\nabla f(S)\|_{L^\infty(0,\tf; L^q(\Omega))}\leq C[G],
\end{align*}
which, together with \eqref{est.narho.1}, leads to \eqref{Est.narho}.

Let us now find another estimate related to $\rho_i/\rho$, $i=1,\ldots,N$. 
From \eqref{Proj_1} and \eqref{Est.B.2} we get
\begin{align*}
\int_0^\tf\int_\Omega\left( \log\frac{\rho_i}{\rho} - \frac{1}{N}\sum_{j=1}^N\log\frac{\rho_j}{\rho} \right)^2 dx dt =
\int_0^\tf\int_\Omega\left( \log\rho_i - \frac{1}{N}\sum_{j=1}^N\log\rho_j \right)^2 dx dt\leq C,
\end{align*}
for $i=1,\ldots,N$. Lemma \ref{lem.aux.1} given below allows us to deduce
\begin{align*}
\left\| \sum_{j=1}^N\log\frac{\rho_j}{\rho} \right\|_{L^2(Q_\tf)}\leq C,
\end{align*}
which, together with the previous estimate, leads to \eqref{est.relchp}.

We conclude by stating a simple algebraic property used in the proof of Prop.\ref{prop.7}.
\begin{lemma}\label{lem.aux.1}
For every $\eps>0$ there exists $C_\eps>0$ such that
\begin{align*}
\sum_{i=1}^N\left(
\log u_i - \frac{1}{N}\sum_{j=1}^N\log u_j\right)^2
\geq C_\eps\left(\sum_{i=1}^N\log u_i \right)^2 - \eps	
\end{align*}
for every $\vec{u}\in (0,\infty)^N$, such that $\sum_{i=1}^N u_i = 1$.	
\end{lemma}
\begin{proof}
By contradiction, assume there exists $\eps_0>0$ such that, for every
$n\in\N$, there is $u^{(n)}\in (0,\infty)^N$ such that
\begin{align*}
\sum_{i=1}^N\left(
\log u_i^{(n)} - \frac{1}{N}\sum_{j=1}^N\log u_j^{(n)}\right)^2
< \frac{1}{n}\left(\sum_{i=1}^N\log u_i^{(n)} \right)^2 - \eps_0,\quad
\sum_{i=1}^N u_i^{(n)} = 1.
\end{align*}
As a consequence $\left(\sum_{i=1}^N\log u_i^{(n)} \right)^2>0$, so
we can define 
$$
v_i^{(n)} = \frac{\log u_i^{(n)}}{ \sum_{k=1}^N\log u_k^{(n)} },\quad
i=1,\ldots,N,
$$
and it follows
\begin{align}\label{aux.1}
\sum_{i=1}^N\left(v_i^{(n)} - \frac{1}{N}\right)^2
< \frac{1}{n} - \frac{\eps_0}{\left(\sum_{i=1}^N\log u_i^{(n)} \right)^2} .
\end{align}
We point out that clearly $u_i^{(n)}\leq 1$ for $i=1,\ldots,N$, which implies that $v_i^{(n)}\geq 0$ for $i=1,\ldots,N$. Furthermore $\sum_{i=1}^N v_i^{(n)} = 1$ by construction, so the sequence $v^{(n)}$ is bounded. Therefore there exists a subsequence (not relabeled) of $v^{(n)}$ that is convergent: $v^{(n)}\to v$ as $n\to\infty$.
Taking the limit $n\to\infty$ in \eqref{aux.1} yields
\begin{align*}
\lim_{n\to\infty}v_i^{(n)} = \frac{1}{N}\quad i=1,\ldots,N,
\quad
\lim_{n\to\infty}\left|\sum_{k=1}^N\log u_k^{(n)} \right| = \infty .
\end{align*}
As a consequence,
\begin{align*}
-\log u_i^{(n)} = v_i^{(n)}\left|\sum_{k=1}^N\log u_k^{(n)} \right|
\to \infty\quad\mbox{as }n\to\infty,~~ i=1,\ldots,N.
\end{align*}
This means that $u_i^{(n)}\to 0$ as $n\to\infty$ for $i=1,\ldots,N$,
which is in contradiction with the fact that $\sum_{i=1}^N u_i^{(n)}=1$.
This finishes the proof of the Lemma.	
\end{proof}


\end{document}